\documentclass[12pt,english]{smfart}

\usepackage{smfthm}
\usepackage{amssymb}

\newcommand{\Qp}{\mathbf{Q}_p}
\newcommand{\Cp}{\mathbf{C}_p}
\newcommand{\Zp}{\mathbf{Z}_p}
\newcommand{\FF}{\mathbf{F}}
\newcommand{\ZZ}{\mathbf{Z}}
\newcommand{\OO}{\mathcal{O}}
\newcommand{\MM}{\mathfrak{m}}
\newcommand{\Qpbar}{\overline{\mathbf{Q}}_p}
\newcommand{\BO}{\mathrm{O}}
\renewcommand{\phi}{\varphi}

\renewcommand{\projlim}{\varprojlim}
\renewcommand{\geq}{\geqslant}
\renewcommand{\leq}{\leqslant} 
\renewcommand{\hat}{\widehat}
\renewcommand{\tilde}{\widetilde}
\newcommand{\calR}{\mathcal{R}}
\newcommand{\calE}{\mathcal{E}}
\newcommand{\calEdag}{\mathcal{E}^\dagger}
\newcommand{\calO}{\mathcal{O}}
\newcommand{\Gal}{\mathrm{Gal}}
\newcommand{\Hom}{\mathrm{Hom}}
\newcommand{\ind}{\mathrm{ind}}

\newcommand{\dR}{\mathrm{dR}}
\newcommand{\cycl}{\mathrm{cyc}}
\newcommand{\End}{\mathrm{End}}
\newcommand{\Frac}{\mathrm{Frac}}
\newcommand{\GL}{\mathrm{GL}}
\newcommand{\Id}{\mathrm{Id}}
\newcommand{\Fil}{\mathrm{Fil}}
\newcommand{\Nm}{\mathrm{N}}
\newcommand{\vp}{\mathrm{val}_p}
\newcommand{\ve}{\mathrm{val}_{\mathbf{E}}}
\newcommand{\unr}{\mathrm{unr}}
\newcommand{\cyc}{\mathrm{cyc}}
\newcommand{\an}{\mathrm{an}}

\newcommand{\bt}{\widetilde{\mathbf{B}}}
\newcommand{\atplus}{\widetilde{\mathbf{A}}^+}
\newcommand{\btplus}{\widetilde{\mathbf{B}}^+}
\newcommand{\etplus}{\widetilde{\mathbf{E}}^+}
\newcommand{\bdr}{\mathbf{B}_{\mathrm{dR}}}  
\newcommand{\btrigplus}{\widetilde{\mathbf{B}}^+_{\mathrm{rig}}} 
\newcommand{\btrig}[2]{\widetilde{\mathbf{B}}^{\dagger #1}_{\mathrm{rig} #2}} 
\newcommand{\atdag}[1]{\widetilde{\mathbf{A}}^{\dagger #1}}
\newcommand{\dcris}{\mathrm{D}_{\mathrm{cris}}}
\newcommand{\dsen}{\mathrm{D}_{\mathrm{Sen}}}
\newcommand{\dfont}{\mathrm{D}}
\renewcommand{\ddag}[1]{\mathrm{D}^{\dagger #1}}
\newcommand{\mfont}{\mathrm{M}}

\newcommand{\dcroc}[1]{[\![ #1 ]\!]}
\newcommand{\LT}{\operatorname{LT}}
\newcommand{\Gm}{\mathbf{G}_{\mathrm{m}}}
\newcommand{\chilt}{\chi_{\LT}}

\theoremstyle{plain}
\newtheorem*{theoA}{Theorem A}
\newtheorem*{theoB}{Theorem B}
\newtheorem*{theoC}{Theorem C}
\theoremstyle{remark}
\newtheorem*{remaA}{Remark}

\author{Laurent Berger}
\address{UMPA ENS de Lyon \\ UMR 5669 du CNRS \\ IUF}
\email{laurent.berger@ens-lyon.fr}
\urladdr{perso.ens-lyon.fr/laurent.berger/}

\date{\today}

\title[Multivariable $(\varphi,\Gamma)$-modules]{Multivariable Lubin-Tate $(\varphi,\Gamma)$-modules \\Êand filtered $\phi$-modules}

\subjclass{11F; 11S; 14G}

\keywords{$(\phi,\Gamma)$-module; Lubin-Tate group; filtered $\phi$-module; crystalline representation; $p$-adic period; Fontaine theory; reflexive sheaf}

\thanks{This research is partially supported by the ANR grant Th\'eHopaD (Th\'eorie de Hodge $p$-adique et D\'evelop\-pements) ANR-11-BS01-005}

\begin{document}

\begin{abstract}
We define some rings of power series in several variables, that are attached to a Lubin-Tate formal module. We then give some examples of $(\varphi,\Gamma)$-modules over those rings. They are the global sections of some reflexive sheaves on the $p$-adic open unit polydisk, that are constructed from a filtered $\phi$-module using a modification process. We prove that we obtain every crystalline $(\varphi,\Gamma)$-module over those rings in this way.
\end{abstract}

\maketitle

\tableofcontents

\setlength{\baselineskip}{18pt}

\section*{Introduction}\label{intro}

Let $F$ be the unramified extension of $\Qp$ of degree $h$ and let $q=p^h$ so that the residue field of $\OO_F$ is $\FF_q$. We fix an embedding $F \subset \Qpbar$ so that if $\sigma : F \to F$ denotes the absolute Frobenius map, which lifts $x \mapsto x^p$ on $\FF_q$, then the $h$ embeddings of $F$ into $\Qpbar$ are given by $\Id,\sigma,\hdots,\sigma^{h-1}$. The symbol $\phi_q$ denotes a $\sigma^h$-semilinear Frobenius map. If $K$ is a subfield of $\Qpbar$, then let $G_K = \Gal(\Qpbar/K)$.

The goal of this article is to present a first attempt to construct some ``multivariable Lubin-Tate $(\phi,\Gamma)$-modules'', that is some $(\phi_q,\Gamma_F)$-modules over rings of power series in $h$ variables, on which $\Gamma_F =\OO_F^\times$ acts by a formula arising from a Lubin-Tate formal $\OO_F$-module. A construction of such $(\phi_q,\Gamma_F)$-modules, but ``in one variable'', was carried out by Kisin and Ren in \cite{KR09}: they prove that in certain cases, the $(\phi_q,\Gamma_F)$-modules arising from Fontaine's standard construction of \cite{F90} are overconvergent. In order to do so, Kisin and Ren adapt the construction of $(\phi,\Gamma)$-modules attached to filtered $(\phi,N)$-modules given in \cite{LB10} to their setting, which allows them to attach a $(\phi_q,\Gamma_F)$-module in one variable to a filtered $\phi_q$-module. They then point out in the introduction of \cite{KR09} that ``it seems likely that in order to obtain a classification valid for any crystalline $G_K$-representation one needs to consider higher dimensional subrings of $W(\mathrm{Fr}\,R)$, constructed using the periods of all the conjugates of [the Lubin-Tate group]''. 

The motivation for these computations is the hope that we can construct some representations of the Borel subgroup of $\GL_2(F)$, for example using the recipe given by Colmez in \cite{PC10}, that would shed some light on the $p$-adic local Langlands correspondence for $\GL_2(F)$ (see \cite{BR10}). Theorems A, B and C below are a very first step in this direction, but remain insufficient. In particular, the ``$p$-adic Fourier theory'' of Schneider and Teitelbaum (see \cite{ST01}) will very likely play an important role in the sequel.

We now describe our results in more detail. Let $\LT_h$ be the Lubin-Tate formal $\OO_F$-module for which multiplication by $p$ is given by $[p](T) =pT+T^q$. We denote by $[a](T)$ the element of $\OO_F \dcroc{T}$ that gives the action of $a \in \OO_F$ on $\LT_h$. We consider two rings $\calR^+(Y)$ and $\calR(Y)$ of power series in the $h$ variables $Y_0,\hdots,Y_{h-1}$, with coefficients in $F$. The ring $\calR^+(Y)$ is the ring of power series that converge on the open unit polydisk, and $\calR(Y)$ is the Robba ring that corresponds to it, by adapting Schneider's construction given in the appendix of \cite{SZ}. The action of the group $\OO_F^\times$ on those rings is given by the formula $a(Y_j) = [\sigma^j(a)](Y_j)$, and the Frobenius map $\phi_q$ is given by $\phi_q(Y_j) = [p](Y_j)$.

The construction of $p$-adic periods for Lubin-Tate groups gives rise to a map $\calR^+(Y) \to \btrigplus$, where $\btrigplus$ is the Fr\'echet completion of $\btplus = W(\etplus)[1/p]$, and we prove (corollary \ref{iotaninj}) that this map is in fact injective (remark: if $\tilde{\calR}^+(Y)$ denotes the completion of the perfection of $\calR^+(Y)$, then the map above extends to a map $\tilde{\calR}^+(Y) \to \btrigplus$ but note that, by the theory of the field of norms of \cite{FW} and \cite{WCN}, this latter map is not injective anymore if $h \geq 2$. This has prevented us from studying \'etale $\phi_q$-modules using Kedlaya's methods, so such considerations are absent from this article).

Let $D$ be a finite dimensional $F$-vector space, endowed with an $F$-linear Frobenius map $\phi_q : D \to D$, and an action of $G_F$ on $D$ that factors through $\Gamma_F$ and commutes with $\phi_q$. For each $0 \leq j \leq h-1$, let $\Fil_j^\bullet$ be a filtration on $D$ that is stable under $\Gamma_F$. 

For example, if $V$ is an $F$-linear crystalline representation of $G_F$ of dimension $d$, then $\dcris(V)$ is a free $F \otimes_{\Qp} F$-module of rank $d$, and we have 
\[ \dcris(V) = D \oplus \phi(D) \oplus \cdots \oplus \phi^{h-1}(D), \]
according to the decomposition of $F \otimes_{\Qp} F$ as $\prod_{\sigma^i : F \to F} F$. Each $\phi^j(D)$ has the filtration induced from $\dcris(V)$, and we set $\Fil_j^k D = \phi^{-j} (\Fil^k  \dcris(V) \cap \phi^j(D))$. 

The composite of the map $\calR^+(Y) \to \btrigplus$ with the map $\phi^{-k} :  \btrigplus \to  \btrigplus$ gives rise to a map $\iota_k : \calR^+(Y) \to \btrigplus$. Let $\log_{\LT}(T)$ be the logarithm of $\LT_h$, and let $\lambda_j = \log_{\LT}(Y_j)/Y_j$ and $\lambda = \prod_{j=0}^{h-1} \lambda_j$ (note that the image of $\prod_{j=0}^{h-1} \log_{\LT}(Y_j)$ in $\btrigplus$ is some $\Qp$-multiple of $t=\log(1+X)$, so that $\lambda$ is an analogue of $t/X$). Define
\[ \mfont^+(D) = \{ y \in \calR^+(Y)[1/\lambda] \otimes_F D,\ \iota_k(y) \in \Fil_{-k}^0(\bdr \otimes_F^{\sigma^{-k}} D) \text{ for all } k \geq h\}. \]
The ring $\calR^+(Y)$ is a Fr\'echet-Stein algebra in the sense of \cite{ST03}, and we therefore have the notion of coadmissible $\calR^+(Y)$-modules, which are the global sections of coherent sheaves on the open unit polydisk.

\begin{theoA}
The module $\mfont^+(D)$ is a reflexive coadmissible $\calR^+(Y)$-module, for all $0 \leq j \leq h-1$, $\mfont^+(D)[\lambda_j/\lambda]$ is a free $\calR^+(Y)[\lambda_j/\lambda]$-module of rank $d$, and we have $\mfont^+(D) = \cap_{j=0}^{h-1} \mfont^+(D)[\lambda_j/\lambda]$.
\end{theoA} 

The definition of $\mfont^+(D)$ is analogous to the one given in \cite{LB10}, \cite{KR09} and similar  articles. When $h=1$, the proof of theorem A relies on the fact that $\mfont^+(D)$ can be seen as a vector bundle on the open unit disk. Our proof of theorem A relies on the one dimensional case, and on the interpretation of $\mfont^+(D)$ as the global sections of a coherent sheaf on the open unit polydisk. 

\begin{remaA} If $h \leq 2$, then $\calR^+(Y)$ is of dimension $\leq 2$ and one can then prove that $\mfont^+(D)$, being reflexive, is actually free of rank $d$ (see remark \ref{h2free}). If $h \geq 3$, I do not know whether $\mfont^+(D)$ is free of rank $d$ in general, nor even if it is finitely generated.
\end{remaA}

Let $\mfont(D) = \calR(Y) \otimes_{\calR^+(Y)} \mfont^+(D)$, so that $\mfont(D)$ is a $(\phi_q,\Gamma_F)$-module over the multivariable Robba ring $\calR(Y)$ (see definition \ref{defpgm}).

\begin{theoB}
If $V$ is an $F$-linear crystalline representation of $G_F$, and if $D$ arises from $\dcris(V)$ as above, then there is a natural map $\btrig{}{} \otimes_{\calR(Y)} \mfont(D) \to \btrig{}{} \otimes_F V$, and this map is an isomorphism.
\end{theoB} 

If $\mfont$ is a $(\phi_q,\Gamma_F)$-module over $\calR(Y)$, then we set $\dcris(\mfont) =  (\calR(Y)[1/t] \otimes_{\calR(Y)} \mfont)^{\Gamma_F}$, and we say that $\mfont$ is crystalline if $\mfont[\lambda_j/\lambda]$ is a free $\calR(Y)[\lambda_j/\lambda]$-module of some rank $d$ for all $j$, $\mfont = \cap_{j=0}^{h-1} \mfont[\lambda_j/\lambda]$, and $\dim \dcris(\mfont) = d$. For example, if $D$ is a filtered $\phi_q$-module with $h$ filtrations $\Fil^\bullet_j$ as above, on which the action of $G_F$ is trivial, then $\mfont(D)$ is a crystalline $(\phi_q,\Gamma_F)$-module.

\begin{theoC}
The functors $\mfont \mapsto \dcris(\mfont)$ and $D \mapsto \mfont(D)$, between the category of crystalline $(\phi_q,\Gamma_F)$-modules over $\calR(Y)$ and the category of $\phi_q$-modules with $h$ filtrations, are mutually inverse.
\end{theoC} 

Note that if $h=1$, then the $(\phi,\Gamma)$-modules that we construct are the classical cyclotomic ones, and theorems A, B and C are well-known.

We now give a short description of the contents of this article: in \S\ref{perlt}, we give some reminders about the $p$-adic periods of Lubin-Tate formal $\OO_F$-modules. In \S\ref{rps}, we define the various rings of power series that we use, and establish some of their properties. In \S\ref{iotanlt}, we embed those rings in the usual rings of $p$-adic periods. In \S\ref{kisinren}, we briefly survey Kisin and Ren's construction and explain why $(\phi_q,\Gamma_F)$-modules over rings of power series in several variables are needed. In \S\ref{constpgm}, we attach such objects to filtered $\phi_q$-modules and prove theorem A. In \S\ref{propsmd}, we prove theorem B. In \S\ref{pmfedp}, we study crystalline $(\phi_q,\Gamma_F)$-modules and prove theorem C.

\section{Periods of Lubin-Tate formal groups}\label{perlt}

Let $\LT_h$ be the Lubin-Tate formal $\OO_F$-module for which multiplication by $p$ is given by $[p](T) =pT+T^q$. We denote by $[a](T)$ the element of $\OO_F \dcroc{T}$ that gives the action of $a \in \OO_F$ on $\LT_h$ and by $S(T,U) = T \oplus U$ the element of $\OO_F \dcroc{T,U}$ that gives addition.

Let $\pi_0 = 0$ and for each $n \geq 1$, let $\pi_n \in \Qpbar$ be such that $[p](\pi_n)=\pi_{n-1}$, with $\pi_1 \neq 0$. We have $\vp(\pi_n) = 1/q^{n-1}(q-1)$ if $n \geq 1$. Let $F_n=F(\pi_n)$ and let $F_\infty = \cup_{n \geq 1} F_n$. Recall that $\Gal(F_\infty/F) \simeq \OO_F^\times$ and that the maximal abelian extension of $F$ is $F_\infty \cdot F^{\unr}$. Denote by $H_F$ the group $\Gal(\Qpbar/F_\infty)$, by $\Gamma_F$ the group $\Gal(F_\infty/F)$ and by $\chi_{\LT}$ the isomorphism $\chi_{\LT} : \Gamma_F \to \OO_F^\times$. In the sequel, we sometimes directly identify $\Gamma_F$ with $\OO_F^\times$, that is we drop ``$\chi_{\LT}$'' from the notation to make the formulas less cumbersome.

Let $\etplus = \projlim_{x \mapsto x^q} \OO_{\Cp}/p$ and $\atplus=W(\etplus)$ denote Fontaine's rings of periods (see \cite{FPP}). Note that we take the limit with respect to the maps $x \mapsto x^q$, which does not change the rings. Let $\phi_q : \atplus \to \atplus$ be given by $\phi_q=\phi^h$. Recall that in \S 9.2 of \cite{C02}, Colmez has constructed a map $\{\cdot\} : \etplus \to \atplus$ having the following property: if $x \in \etplus$, then $\{x\}$ is the unique element of $\atplus$ that lifts $x$ and satisfies $\phi_q(\{x\}) = [p](\{x\})$. Let $\theta : \atplus \to \OO_{\Cp}$ denote Fontaine's map (see \cite{FPP}). If $x=(x_0,x_1,\hdots)$, then $\theta(\{x\}) = \lim_{n \to \infty} [p^n](\hat{x}_n)$, where $\hat{x}_n \in \OO_{\Cp}$ is any lift of $x_n$.

Let $u=\{(\overline{\pi}_0,\overline{\pi}_1,\hdots)\} \in \atplus$, so that $g(u) = [\chilt(g)](u)$ if $g \in G_F$. 

Let $\log_{\LT}(T) \in F\dcroc{T}$ denote the Lubin-Tate logarithm map, which converges on the  open unit disk and satisfies $\log_{\LT}([a](T)) = a  \cdot \log_{\LT}(T)$ if $a\in \OO_F$. Recall (see \S 9.3 of \cite{C02}) that $\log_{\LT}(u)$ converges in $\btrigplus$ to an element $t_F$ which satisfies $g(t_F)=\chilt(g) \cdot t_F$.

Let $Q_k(T)$ be the minimal polynomial of $\pi_k$ over $F$. We have $Q_0(T)=T$, $Q_1(T)=p+T^{q-1}$ and $Q_{k+1}(T)=Q_k([p](T))$ if $k \geq 1$. Note that 
\[ \log_{\LT}(T) = T \cdot \prod_{k \geq 1} \frac{Q_k(T)}{p}. \] 
Indeed,  $\log_{\LT}(T) = \lim_{k \to \infty} p^{-k} \cdot [p^k](T)$ (\S 9.3 of \cite{C02}) and $[p^k](T) = Q_0(T) \cdots Q_k(T)$. Let $\exp_{\LT}(T)$ denote the inverse of $\log_{\LT}(T)$. We have $\exp_{\LT}(T) = \sum_{k=1}^\infty e_k T^k$ with $v_p(e_k) \geq -k/(q-1)$. For example, $\log_{\Gm}(T)=\log(1+T)$ and $\exp_{\Gm}(T)=\exp(T)-1$.

\begin{rema}\label{chgfrob}
Our special choice of $[p](T)=pT+T^q$ is the simplest. Since $[p](T)$ belongs to $\Zp[T]$, the series $Q_k(T)$, $\log_{\LT}(T)$ and $\exp_{\LT}(T)$ all have coefficients in $\Qp$. It also implies that $[\sigma(a)] (T)  = \sigma( [a](T) )$, since $[a](T) = \exp_{\LT}(a \cdot \log_{\LT}(T))$.
\end{rema}

\begin{lemm}\label{lielt}
If $z \in \MM_{\Cp}$, then 
\[ \frac{[1+a](z)-z}{a} = \log_{\LT}(z) \cdot \frac{dS}{dU} (z,0) + \BO(a), \]
as $a \to 0$ in $\OO_F$.
\end{lemm}

\begin{proof}
We are looking at the limit of $(S(z,[a](z))-z)/a$ as $a \to 0$. If $a$ is small enough, then $[a](z) = \exp_{\LT}(a \cdot\log_{\LT}(z)) = a \cdot\log_{\LT}(z) + \BO(a^2)$, which implies the lemma.
\end{proof}

\section{Rings of multivariable power series}\label{rps}

We consider power series in the $h$ variables $Y_0,\hdots,Y_{h-1}$. If $Y^m = Y_0^{m_0} \cdots Y_{h-1}^{m_{h-1}}$ is a monomial, then its weight is $w(m) = m_0 + p m_1 + \cdots + p^{h-1} m_{h-1}$. If $I$ is a subinterval of $[0;+\infty]$ and if $J = \{ j_1, \hdots, j_k \}$ is a subset of $\{0,\hdots,h-1\}$, then (adapting Appendix A of \cite{SZ} to our situation) we define $\calR^I (\{Y_j\}_{j \in J})$ to be the ring of power series  \[ f(Y_{j_1},\hdots,Y_{j_k}) = \sum_{m_1,\hdots,m_k \in \ZZ} a_{m_1 \hdots m_k} Y_{j_1}^{m_1} \cdots Y_{j_k}^{m_k}, \] 
such that $\vp(a_m) + w(m)/r \to +\infty$ for all $r \in I$. In other words, $f(Y)$ is required to converge on the polyannulus $\{ (Y_0,\hdots,Y_{h-1})$ such that $|Y_0|=p^{-1/r}$, \dots, $|Y_{h-1}|=p^{-p^{h-1}/r} \}$ for all $r \in I$. We then define $W(f(Y),r) = \inf_{m \in \ZZ} (\vp(a_m) + w(m)/r)$ and, if $I$ is closed, $W(f(Y),I) = \inf_{r \in I} W(f(Y),r)$.

We let $\calR^+ (\{Y_j\}_{j \in J}) = \calR^{[0;+\infty[} (\{Y_j\}_{j \in J})$ be the ring of holomorphic functions on the open unit polydisk corresponding to $J$. The Robba ring $\calR(\{Y_j\}_{j \in J})$ is defined as $\calR(\{Y_j\}_{j \in J})  = \cup_{r  \geq 0} \calR^{[r;+\infty[} (\{Y_j\}_{j \in J})$. In order to lighten the notation, we write $\calR^I(Y)$, $\calR^+(Y)$ and $\calR(Y)$ instead of $\calR^I(Y_0,\hdots,Y_{h-1})$, $\calR^+(Y_0,\hdots,Y_{h-1})$ and $\calR(Y_0,\hdots,Y_{h-1})$.

The rings $\calR^I (\{Y_j\}_{j \in J})$ are endowed with an $F$-linear action of $\Gamma_F$, given by the formula $a(Y_j) = [\sigma^j(a)](Y_j)$. There is also an $F$-linear Frobenius map :
\[ \phi_q : \calR^I (\{Y_j\}_{j \in J}) \to \calR^{I'} (\{Y_j\}_{j \in J}), \] 
given by $Y_j \mapsto [p](Y_j)$, for appropriate $I$ and $I'$.

On the ring $\calR^I(Y)$, we can define in addition an absolute $\sigma$-semilinear Frobenius map $\phi$ by $Y_j \mapsto Y_{j+1}$ for $0 \leq j \leq h-2$ and $Y_{h-1} \mapsto [p](Y_0)$. This map $\phi$ has the property that $\phi^h = \phi_q$, and it also commutes with $\Gamma_F$.

Let $t_i = \log_{\LT}(Y_i)$. Since $a(Y_i) = [\sigma^i(a)](Y_i)$ if $a \in \Gamma_F$, we have $a(t_i) = \sigma^i(a) \cdot t_i$ so that $g(t_0 \cdots t_{h-1}) = \Nm_{F/\Qp}(\chilt(g)) \cdot t_0 \cdots t_{h-1} = \chi_{\cycl}(g) \cdot t_0 \cdots t_{h-1}$ if $g \in G_F$ as well as $\phi(t_0 \cdots t_{h-1})=p \cdot t_0 \cdots t_{h-1}$.  The element $t_0 \cdots t_{h-1}$ therefore behaves like a $\Qp$-multiple of the ``usual'' $t$ of $p$-adic Hodge theory (see proposition \ref{imgt} for a more precise statement).

The following two propositions are variations on the ``Weierstrass division theorem''.

\begin{prop}\label{weierdiv}
Let $I=[0;s]$ or $[0;s[$ and let $P(T) \in \OO_F[T]$ be a monic polynomial of degree $d$ whose nonleading coefficients are all divisible by $p$. If $f \in \calR^I (\{Y_j\}_{j \in J})$, then there exists $g \in \calR^I (\{Y_j\}_{j \in J})$ and $f_0,\hdots,f_{d-1} \in  \calR^I (\{Y_j\}_{j \in J \setminus \{i\}})$ such that 
\[ f = f_0 + f_1 Y_i + \cdots + f_{d-1} Y_i^{d-1} + g \cdot P(Y_i). \]
\end{prop}

\begin{proof}
If $I=[0;s]$ is closed, then this is a straightforward consequence of the Weierstrass division theorem. Since $g$ and the $f_i$'s are uniquely determined, the result extends to the case when $I=[0;s[$.
\end{proof}

\begin{prop}\label{weierdivcirc}
Let $I=[s;s]$ and let $P(T) \in \OO_F[T]$ be a monic polynomial of degree $d$, all of whose roots are of valuation $-1/s$. If $f \in \calR^I (\{Y_j\}_{j \in J})$, then there exists $g \in \calR^I (\{Y_j\}_{j \in J})$ and $f_0,\hdots,f_{d-1} \in  \calR^I (\{Y_j\}_{j \in J \setminus \{i\}})$ such that 
\[ f = f_0 + f_1 Y_i + \cdots + f_{d-1} Y_i^{d-1} + g \cdot P(Y_i). \]
\end{prop}

\begin{proof}
The polynomial $Q(T)=P(1/T)T^d/P(0)$ is monic and all its roots are of valuation $1/s$. Write $f=f^+ + f^-$ where $f^+$ contains positive powers of $Y_i$ and $f^-$ contains negative powers of $Y_i$. One may Weierstrass divide $f^+$ by $P(Y_i)$ and $f^-$ by $Q(1/Y_i)$, which implies the proposition.
\end{proof}

\begin{lemm}\label{qpanry}
The action of $\Gamma_F$ on $\calR^I(Y)$ is locally $\Qp$-analytic, and we have
\[ [1+a](f(Y)) = f(Y) + \sum_{j=0}^{h-1} \sigma^j(a) \cdot \log_{\LT}(Y_j) \cdot \frac{dS}{dU}(Y_j,0) \cdot \frac{df}{dY_j} (Y) + \BO(a^2). \]
\end{lemm}

\begin{proof}
The above formula follows from the fact that $[1+a](Y_j) = Y_j \oplus [a](Y_j) = Y_j \oplus (\sigma^j(a) \cdot \log_{\LT}(Y_j) + \BO(a^2))$. 
\end{proof}

\begin{prop}\label{injeval}
Let $\rho=(\rho_1,\hdots,\rho_{h-1})$ and let $\calR^\rho_{F_k} (T_1,\hdots,T_{h-1})$ denote the ring of Laurent series converging for $|T_i| = \rho_i$,  with coefficients in $F_k$. If the $z_i \in \MM_{\hat{F}_\infty}$ are such that $\log_{\LT}(z_i) \neq 0$, $|z_i| = \rho_i$ and $g(z_i) = [\sigma^i(g)] (z_i)$ for $g \in \OO_F^\times$, then the map $\calR^\rho_{F_k} (T_1,\hdots,T_{h-1}) \to \Cp$ given by evaluating at $(z_1,\hdots,z_{h-1})$ is injective.
\end{prop}

\begin{proof}
Suppose that $f(z_1,\hdots,z_{h-1})=0$ for some $f \in \calR^\rho_{F_k} (T_1,\hdots,T_{h-1})$. If $g \in \Gamma_{F_k}$, then $f(g(z_1),\hdots,g(z_{h-1}))=0$. If $g=1+a$ with $a$ small, then lemma  \ref{lielt} provides us with $h-1$ elements $y_1,\hdots,y_{h-1}$ of $\hat{F}_\infty$ such that $g(z_i) = z_i + \sigma^i(a) \cdot y_i+\BO(a^2)$. Since $y_i=\log_{\LT}(z_i) \cdot dS/dU(z_i,0)$ and $dS/dU$ is a unit and $\log_{\LT}(z_i) \neq 0$, the elements $y_1,\hdots,y_{h-1}$ are all nonzero.

If $f \neq 0$ and $m$ is the smallest index for which $f$ has a nonzero partial derivative of order $m$ at $(z_1,\hdots,z_{h-1})$ and if we expand $f(g(z_1),\hdots,g(z_{h-1}))$ around $(z_1,\hdots,z_{h-1})$ (which generalizes lemma \ref{qpanry}), then we get
\begin{multline*} \sum_{j_1+\cdots+j_{h-1}=m} (\sigma^1(a)y_1)^{j_1} \cdots (\sigma^{h-1}(a)y_{h-1})^{j_{h-1}} \frac{d^m f}{dT_1^{j_1} \cdots  dT_{h-1}^{j_{h-1}}} (z_1,\hdots,z_{h-1}) \\ + \BO(a^{m+1}). 
\end{multline*}

Since $f(g(z_1),\hdots,g(z_{h-1}))=0$, the above linear combination is a homogeneous polynomial, of degree $m$ in $h-1$ variables and coefficients in $\hat{F}_\infty$, that is identically zero on $(\sigma^1(a),\hdots,\sigma^{h-1}(a))$. The shortest nonzero polynomial that is identically zero on $(\sigma^1(a),\hdots,\sigma^{h-1}(a))$ can be taken to have coefficients in $F$ and Artin's theorem on the algebraic independence of characters implies that it is equal to zero. Since all the $y_i$'s are nonzero, all the partial derivatives of order $m$ of $f$ are zero, so that finally $f=0$.
\end{proof}

\section{Embeddings in $\bdr$}\label{iotanlt}

We now explain how to embed the rings of power series of the previous section in the usual rings of $p$-adic periods. Let $\bt^I$ be the ring defined in \S 2.1 of \cite{LB2}. This ring is complete with respect to the valuation $V(\cdot,I)$ (denoted by $V_I(\cdot)$ in \S 2.1 of ibid.). Recall that if $x=\sum_{k \geq 0} p^k [x_k] \in \atplus$, then $V(x,r)=\inf_k (\ve(x_k)+krp/(p-1))$. Set $r_F = p^{h-1} \cdot q/(q-1) \cdot (p-1)/p$ (for example, $r_{\Qp}=1$ and if $h>1$, then $r_F < p^{h-1}$).

\begin{prop}\label{valphiu}
If $r \geq r_F$ and $m \in \ZZ$, then $V(\phi^j(u)^m,r) = m \cdot p^j \cdot q/(q-1)$ for $0 \leq j \leq h-1$.
\end{prop}

\begin{proof}
Recall that $u=\{\pi\}$ where $\pi=(\pi_0,\pi_1,\hdots)$ with $\vp(\pi_n)=1/q^{n-1}(q-1)$ for $n \geq 1$, so that $\ve(\pi)=q/(q-1)$. We have $\phi^j(u)=[\pi^{p^j}]+\sum_{k \geq 1} p^k [u_{k,j}]$ where $\ve(u_{k,j}) > 0$, so that if $r \geq r_F$, then $\phi^j(u) / [\pi^{p^j}]$ is a unit of $\atdag{,r}$ and the proposition follows.
\end{proof}

Note that a better estimate on the $\ve(u_{k,j})$ would allow us to take a smaller value for $r_F$. Let $s_n = p^{n-h}(q-1)$ and let $r_n=p^{n-1}(p-1)$ (so that $s_n \cdot q/(q-1) = r_n \cdot p/(p-1)$).

\begin{prop}\label{embti}
If $n \geq h$, and if $f(Y) \in \calR^{[s_n;s_n]} (Y)$, then $f(u,\hdots,\phi^{h-1}(u))$ converges in $\bt^{[r_n;r_n]}$.
\end{prop}

\begin{proof}
If $f(Y) = \sum_{m\in \ZZ^h} a_m Y^m \in \calR^{[s_n;s_n]} (Y)$, then $\vp(a_m) + w(m) / (p^{n-h}(q-1)) \to +\infty$. If $n \geq h$, then $r_n > r_F$ so that $V(\phi^j(u)^{m_j},r) = m_j \cdot p^j \cdot q/(q-1)$ for $0 \leq j \leq h-1$ by proposition \ref{valphiu}, and then
\[ V(a_{m_0,\hdots,m_{h-1}} u^{m_0} \cdots \phi^{h-1}(u)^{m_{h-1}},r_n) \to +\infty. \]
The series $f(u,\hdots,\phi^{h-1}(u))$ therefore converges in $\bt^{[r_n;r_n]}$.
\end{proof}

\begin{coro}\label{embtplus}
If $n \geq h$, and if $f(Y) \in \calR^{[0;s_n]} (Y)$, then $f(u,\hdots,\phi^{h-1}(u))$ converges in $\bt^{[0;r_n]}$. If $f(Y) \in \calR^+ (Y)$, then $f(u,\hdots,\phi^{h-1}(u))$ converges in $\btrigplus$.
\end{coro}

\begin{proof}
If $f \in \calR^{[0;s_n]} (Y)$, then each term of the series $f(u,\hdots,\phi^{h-1}(u))$ belongs to $\btplus$ so that it converges in $\bt^{[0;r_n]}$ by the maximum modulus principle (corollary 2.20 of \cite{LB2}). The second assertion follows by passing to the limit.
\end{proof}

The image of $\log_{\LT}(Y_0) \cdots \log_{\LT}(Y_{h-1})$ in $\btrigplus \subset \bdr^+$ is $a \cdot t$ with $a \in \Qp$, as we have seen above. We henceforth denote by $t$ the element of $\calR^+(Y)$ whose image in $\btrigplus$ is $t$, that is $t = \log_{\LT}(Y_0) \cdots \log_{\LT}(Y_{h-1})/a$. In the following proposition, we determine the valuation of $a$ (this is not used in the rest of this article).

\begin{prop}\label{imgt}
In the ring $\bdr^+$, the product $\log_{\LT}(u) \cdots \log_{\LT}(\phi^{h-1}(u))$ belongs to $p^{h-1} \cdot \Zp^\times \cdot t$, where $t$ is the usual $t$ of $p$-adic Hodge theory.
\end{prop}

\begin{proof}
We have seen that $\log_{\LT}(u) \cdots \log_{\LT}(\phi^{h-1}(u)) = a \cdot t$ with $a \in \Qp$, and we now compute $\vp(a)$. We have $\log_{\LT}(u) = u \cdot \prod_{k \geq 1} Q_k(u)/p$ and likewise, if $\pi=[\varepsilon]-1$, then $t=\pi \cdot \prod_{k \geq 1} Q^{\cyc}_k(\pi)/p$. This implies that $\theta(t/\log_{\LT}(u)) = \theta(\pi/u)$. Since both $\pi/\phi^{-1}(\pi)$ and $u/\phi_q^{-1}(u)$ are generators of $\ker(\theta)$ in $\atplus$, we have $\vp(\theta(t/\log_{\LT}(u))) = 1/(p-1) - 1/(q-1)$. On the other hand, $\vp(\theta \circ \phi^j(u)) = \vp(\lim_{n \to \infty} [p^n](\pi_n^{p^j})) = 1+ p^j/(q-1)$ if $1 \leq j \leq h-1$, so that $\vp(\theta(\log_{\LT}(\phi^j(u)))) = 1+ p^j/(q-1)$. This implies that $\vp(a)=\vp(\theta(a))=h-1$, and hence the proposition.
\end{proof}

\begin{defi}\label{defiotan}
Let $\iota_n : \calR^{[s_n;s_n]} (Y) \to \bdr^+$ be the compositum of the map defined above, with the map $\phi^{-n} : \bt^{[r_n;r_n]} \to \bt^{[r_0;r_0]}$ and the map $\bt^{[r_0;r_0]} \subset \bdr^+$ defined in \S 2.2 of \cite{LB2}.
\end{defi}

It follows from the definition as well as the formulas for $\phi$ and the action of $\Gamma_F$ on $\calR^I (Y)$ that $\iota_{n+1}(\phi(f)) = \iota_n(f)$ when applicable, and that $g(\iota_n(f)) = \iota_n(g(f))$ if $g \in G_F$. Since $\iota_n(t) = p^{-n} t$, we can extend $\iota_n$ to $\iota_n : \calR^{[s_n;s_n]} (Y)[1/t] \to \bdr$.

\begin{theo}\label{embdr}
If $n \geq h$, if $f \in \calR^{[s_n;s_n]} (Y)$, and if $n=hk+i$ with $0 \leq i \leq h-1$, then we have $\iota_n(f) \in \Fil^1 \bdr^+$ if and only if $f \in Q_k(Y_i) \cdot \calR^{[s_n;s_n]} (Y)$.
\end{theo}

\begin{proof}
Recall that $u =\{ (\pi_0, \pi_1, \hdots) \} \in \atplus$. If $m \geq 1$ and $u_m = \theta(\phi^{-m}(u)) \in \hat{F}_\infty$, then $g(u_m) = [\sigma^{-m}(g)](u_m)$. Note that if $m=h\ell$, then $u_m = \theta(\phi_q^{-\ell}(u)) = \pi_\ell$. The theorem is equivalent to the assertion that $f^{\sigma^{-n}}(u_n,\hdots,u_{n-h+1})=0$ in $\Cp$ if and only if $f \in Q_k(Y_i) \cdot \calR^{[s_n;s_n]} (Y)$. We have $u_{n-i} = \pi_k$ so that if $f$ belongs to $Q_k(Y_i) \cdot \calR^{[s_n;s_n]} (Y)$, then $f^{\sigma^{-n}}(u_n,\hdots,u_{n-h+1})=0$.

Since $Q_k(T)$ is a monic polynomial of degree $d=q^{k-1}(q-1)$, whose nonleading coefficients are divisible by $p$, we can use proposition \ref{weierdivcirc} to write $f^{\sigma^{-n}} =f_0 + Y_i f_1 + \cdots + Y_i^{d-1} f_{d-1} + Q_k(Y_i) r$ with $f_i$ a power series in the $Y_j$'s with $j \neq i$. Proposition \ref{injeval} applied to $f_0 + \pi_k f_1 + \cdots + \pi_k^{d-1} f_{d-1}$, with the $T_j$'s a suitable permutation of the $Y_j$'s, shows that $f_0 + \pi_k f_1 + \cdots + \pi_k^{d-1} f_{d-1}=0$ . Therefore, $f  = Q_k(Y_i) r^{\sigma^n}$, which proves the theorem.
\end{proof}

\begin{coro}\label{iotaninj}
If $n \geq h$, then the map $\iota_n : \calR^{[s_n;s_n]} (Y) \to \bdr^+$ is injective. If $n \in \ZZ$, then the map $\iota_n : \calR^+ (Y) \to \bdr^+$ is injective.
\end{coro}

\begin{proof}
The first assertion follows from theorem \ref{embdr}. The second follows from that, and from the fact that $\iota_{n+1}(\phi(f)) = \iota_n(f)$ for the other $n$.
\end{proof}

\section{$(\phi_q,\Gamma_F)$-modules in one variable}\label{kisinren}

Before constructing $(\phi_q,\Gamma_F)$-modules over $\calR(Y)$, we review Kisin and Ren's construction of $(\phi_q,\Gamma_F)$-modules in one variable and explain why we need rings in several variables.

Let $Y_0$ be the variable of \S\ref{rps}, and let $\calE(Y_0)$ be Fontaine's field of \cite{F90} with coefficients in $F$, that is $\calE(Y_0) = \calO_{\calE}(Y_0)[1/p]$ where $\calO_{\calE}(Y_0)$ is the $p$-adic completion of $\OO_F\dcroc{Y_0}[1/Y_0]$. We let $\calEdag(Y_0)$ and $\calR(Y_0)$ denote the corresponding overconvergent and Robba rings. If $I$ is a subinterval of $[0;+\infty]$, then we denote as above by $\calR^I(Y_0)$ the set of power series $f(Y_0) = \sum_{m \in \ZZ} a_m Y_0^m$ that belong to $\calR^I(Y_0,\hdots,Y_{h-1})$ via the natural inclusion. 

If $K/F$ is a finite extension, then by the theory of the field of norms (see \cite{FW} and \cite{WCN}), there corresponds to it a finite extension $\calE_K(Y_0)$ of $\calE(Y_0)$, of degree $[K_\infty:F_\infty]$. A $(\phi_q,\Gamma_K)$-module over $\calE_K(Y_0)$ is a finite dimensional $\calE_K(Y_0)$-vector space $\dfont$, along with a semilinear $\phi_q$ and a compatible action of $\Gamma_K$. We say that $\dfont$ is \'etale if $\dfont = \calE_K(Y_0) \otimes_{\calO_{\calE_K}(Y_0)} \dfont_0$ where $\dfont_0$ is a $(\phi_q,\Gamma_K)$-module over $\calO_{\calE_K}(Y_0)$. By specializing the constructions of \cite{F90}, Kisin and Ren prove the following theorem in their paper (theorem 1.6 of \cite{KR09}).

\begin{theo}\label{fontkisren}
The functors 
\[ V \mapsto (\hat{\calE}(Y_0)^{\unr} \otimes_F V)^{H_K} \text{ and } \dfont \mapsto (\hat{\calE}(Y_0)^{\unr} \otimes_{\calE_K(Y_0)} \dfont)^{\phi_q=1} \] 
give rise to mutually inverse equivalences of categories between the category of $F$-linear representations of $G_K$ and the category of \'etale $(\phi_q,\Gamma_K)$-modules over $\calE_K(Y_0)$.
\end{theo}

We say that an $F$-linear representation of $G_K$ is $F$-analytic if it is Hodge-Tate with weights $0$ (i.e.\ $\Cp$-admissible) at all embeddings $\tau \neq \Id$. Kisin and Ren then go on to show that if $K \subset F_\infty$, and if $V$ is a crystalline representation of $G_K$, that is $F$-analytic, then the $(\phi_q,\Gamma_K)$-module attached to $V$ is overconvergent (see \S 3.3 of ibid.).

Assume from now on that $K \subset F_\infty$, so that $\calE_K(Y_0) = \calE(Y_0)$. If $\dfont$ is a $(\phi_q,\Gamma_K)$-module over $\calR(Y_0)$, and if $g \in \Gamma_K$ is close enough to $1$, then by standard arguments (see \S 4.1 of \cite{LB2} or \S 2.1 of \cite{KR09}), the series $\log(g) = \log(1+(g-1))$ gives rise to a differential operator $\nabla_g : \dfont \to \dfont$. The map $\mathrm{Lie}\,\Gamma_F \to \End(\dfont)$ arising from $v \mapsto \nabla_{\exp(v)}$ is $\Qp$-linear, and we say that $\dfont$ is $F$-analytic if this map is $F$-linear (see \S 2.1 of \cite{KR09} and \S 1.3 of \cite{FX12}). This is equivalent to the requirement that $\nabla_j=0$ on $\dfont$ for $1 \leq j \leq h-1$, where $\nabla_j$ is the partial derivative in the direction $\sigma^j$.

\begin{theo}\label{mex}
If $V$ is an overconvergent $F$-linear representation of $G_K$, and if $\dfont(V) = \calR(Y_0) \otimes_{\calEdag(Y_0)} \ddag{}(V)$, then $\dfont(V)$ is $F$-analytic if and only if $V$ is $F$-analytic.
\end{theo}

\begin{proof}
Choose $1 \leq j \leq h-1$, and take $n \gg 0$ such that $n = -j \bmod{h}$. By proposition \ref{embti}, we have a map $\theta \circ \phi^{-n} : \calR^{[s_n;s_n]}(Y_0) \to \bdr^+ \to \Cp$, giving rise to an isomorphism 
\[ \Cp \otimes^{\theta \circ \phi^{-n}}_{\calR^{[s_n;s_n]}(Y_0)} \dfont^{[s_n;s_n]}(V) \to \Cp \otimes^{\sigma^j}_F V. \]

We first prove that if $\dfont(V)$ is $F$-analytic, then $V$ is $\Cp$-admissible at the embedding $\sigma^j$. Let $\hat{F}_\infty^{(j)}$ denote the field of locally $\sigma^j$-analytic vectors of $\hat{F}_\infty$ for the action of $\Gamma_K$. Note that $\theta \circ \phi^{-n}( \calR^{[s_n;s_n]}(Y_0) ) \subset \hat{F}_\infty^{(j)}$. Let $\dsen^{(j)}(V)$ be the $\hat{F}_\infty^{(j)}$-vector space 
\[ \dsen^{(j)}(V) = \hat{F}_\infty^{(j)} \otimes_{\theta \circ \phi^{-n}( \calR^{[s_n;s_n]}(Y_0) )} \theta \circ \phi^{-n}( \dfont^{[s_n;s_n]}(V) ). \]
It is of dimension $d$, its image in $(\Cp \otimes^{\sigma^j}_F V)^{H_F}$ generates $\Cp \otimes^{\sigma^j}_F V$, and its elements are all locally $\sigma^j$-analytic vectors of $(\Cp \otimes^{\sigma^j}_F V)^{H_F}$ because $\dfont(V)$ is $F$-analytic and $\phi^{-n} \circ \nabla_j = \nabla_0 \circ \phi^{-n}$. If $y \in \dsen^{(j)}(V)$, then $(g(y)-y) / (\sigma^j \circ \chilt(g)-1)$ has a limit as $g \to 1$, and we call $\nabla_j(y)$ this limit. We then have $g(y) = \exp(\log_p (\sigma^j \circ \chilt(g)) \cdot \nabla_j)(y)$ if $g \in \Gamma_K$ is close to $1$. 

Recall that there exists $a_j \in \Cp$ such that $\log_p(\sigma^j \circ \chilt(g)) = g(a_j)-a_j$. For example, one can take $a_j = \log_p (\theta \circ \iota_0(t_j))$. The element $a_j$ then belongs to $\hat{F}_\infty^{(j)}$ for obvious reasons and satisfies $\nabla_j(a_j)=1$. Take $y \in \dsen^{(j)}(V)$, and choose $a_{j,0} \in F_\infty$  such that $|a_j-a_{j,0}|_p$ is small enough. The series
\[ C(y) = \sum_{k \geq 0} (-1)^k \frac{(a_j-a_{j,0})^k}{k!} \nabla_j^k(y) \]
then converges for the topology of $\dsen^{(j)}(V)$ (the technical details concerning convergence in such spaces of locally analytic vectors can be found in \cite{STLAV}) and a short computation shows that $\nabla_j(C(y)) = 0$, so that $C(y) \in (\Cp \otimes^{\sigma^j}_F V)^{G_{F_n}}$ for some $n = n(y) \gg 0$. In addition, $n(y)=n(\nabla_j^k(y))$ for $k \geq 0$, the series for $C(\nabla_j^k(y))$ also converges for the topology of $\dsen^{(j)}(V)$, and $y= \sum_{k \geq 0} (a_j-a_{j,0})^k/k! \cdot C(\nabla_j^k(y))$.

If $y_1,\hdots,y_d$ is a basis of $\dsen^{(j)}(V)$, and if $n \geq \max n(y_i)$, then the above computations show that the elements $y_i$ belong to $\hat{F}_\infty^{(j)} \otimes_{F_n} (\Cp \otimes^{\sigma^j}_F V)^{G_{F_n}}$, so that $(\Cp \otimes^{\sigma^j}_F V)^{G_{F_n}}$ generates $(\Cp \otimes^{\sigma^j}_F V)^{H_F}$. This implies that $V$ is $\Cp$-admissible at the embedding $\sigma^j$. This is true for all $1 \leq j \leq h-1$, and therefore $V$ is $F$-analytic.

We now prove that if $V$ is $\Cp$-admissible at the embedding $\sigma^j$, then $\nabla_j = 0$ on $\dfont(V)$. Choose $n=hm-j$ with $m \gg 0$. Since $j \neq 0 \bmod{h}$, the map $\theta \circ \phi^{-n} : \calR^{[s_n;s_n]}(Y_0) \to \Cp$ is injective by theorem \ref{embdr}. This implies that the map 
\[ \dfont^{[s_n;s_n]}(V) \to  \Cp \otimes^{\theta \circ \phi^{-n}}_{\calR^{[s_n;s_n]}(Y_0)} \dfont^{[s_n;s_n]}(V) \] is injective, and hence the map $\dfont^{[s_n;s_n]}(V) \to \Cp \otimes_F^{\sigma^j} V$ is also injective. Therefore, we have an injection $\dfont^{[s_n;s_n]}(V) \to ((\Cp \otimes_F^{\sigma^j} V)^{H_F})^{\an}$ where $((\Cp \otimes_F^{\sigma^j} V)^{H_F})^{\an}$ denotes the set of locally $\Qp$-analytic vectors of $(\Cp \otimes_F^{\sigma^j} V)^{H_F}$. If $V$ is $\Cp$-admissible at the embedding $\sigma^j$, then $((\Cp \otimes_F^{\sigma^j} V)^{H_F})^{\an} = (\hat{F}^{\an}_\infty)^d$. One of the main results of \cite{STLAV} is that $\nabla_0 = 0$  on $\hat{F}^{\an}_\infty$ (it is shown in \cite{STLAV} that, in a suitable sense, $\hat{F}^{\an}_\infty$ is generated by $F_\infty$ and the elements $a_1,\hdots,a_{h-1}$). This implies that $\nabla_j = 0$ on $\dfont^{[s_n;s_n]}(V)$, since $\phi^{-n} \circ \nabla_j = \nabla_0 \circ \phi^{-n}$.
\end{proof}

Note that an analogous argument for the proof of the implication ``$\dfont(V)$ is $F$-analytic implies $V$ is $F$-analytic'' in certain cases was given by Bingyong Xie.

\begin{coro}\label{colmabsirr}
If $V$ is an absolutely irreducible $F$-linear overconvergent representation of $G_K$, then there exists a character $\delta$ of $\Gamma_K$ such that $V \otimes \delta$ is $F$-analytic.
\end{coro}

\begin{proof}
We give a sketch of the proof. Choose some $g \in \Gamma_K$ such that $\log_p(\chilt(g)) \neq 0$, and let $\nabla = \log(g) / \log_p(\chilt(g))$. Choose $r>0$ large enough and $s \geq qr$. If $a \in \OO_F$, and if $\vp(a) \geq n$ for $n=n(r,s)$ large enough, then the series $\exp(a \cdot \nabla)$ converges to an operator on $\dfont^{[r;s]}(V)$. This way, we can define a twisted action of $\Gamma_{K_n}$ on $\dfont^{[r;s]}(V)$, by the formula $h \star x = \exp(\log_p(\chilt(h)) \cdot \nabla)(x)$. This action is now $F$-analytic by construction.

Since $s \geq qr$, the modules $\dfont^{[q^m r;q^m s]}(V)$ for $m \geq 0$ are glued together by $\phi_q$ and this way, we get a new action of $\Gamma_{K_n}$ on $\dfont(V)$. Since $\phi_q$ is unchanged, this new $(\phi_q, \Gamma_{K_n})$-module is \'etale, and therefore corresponds to a representation $W$ of $G_{K_n}$. This representation $W$ is $F$-analytic by theorem \ref{mex}, and its restriction to $H_F$ is isomorphic to $V$.

The space $\Hom(V, \ind_{G_{K_n}}^{G_K} W)^{H_F}$ is nonempty, and is a finite dimensional representation of $\Gamma_K$. Since $\Gamma_K$ is abelian, we find (possibly extending scalars) a character $\delta$ of $\Gamma_K$ and a nonzero $f \in \Hom(V, \ind_{G_{K_n}}^{G_K} W)^{H_F}$ such that $h(f)=\delta(h) \cdot f$ if $h \in G_K$. This $f$ gives rise to a nonzero $G_K$-equivariant map $V \otimes \delta \to \ind_{G_{K_n}}^{G_K} W$. Since $\ind_{G_{K_n}}^{G_K} W$ is $F$-analytic and $V$ is absolutely irreducible, the corollary follows.
\end{proof}

Corollary \ref{colmabsirr} (as well as theorem 0.6 of \cite{FX12}) suggests that if we want to attach overconvergent $(\phi_q,\Gamma_K)$-modules to all $F$-linear representations of $G_K$, then we need to go beyond the objects in only one variable. We finish with a conjecture that seems reasonable enough, since it holds for crystalline representations by the work of Kisin and Ren (see also theorem 0.3 of \cite{FX12}).

\begin{conj}\label{locanalone}
If $V$ is $F$-analytic, then it is overconvergent.
\end{conj}

\section{Construction of $\calR^+(Y)$-modules}\label{constpgm}

We now explain how to construct some $\calR^+(Y)$-modules $\mfont^+(D)$ that are attached to some filtered $\phi_q$-modules $D$. Let $D$ be a finite dimensional $F$-vector space, endowed with an $F$-linear Frobenius map $\phi_q : D \to D$, and an action of $G_F$ on $D$ that factors through $\Gamma_F$ and commutes with $\phi_q$. 

If $n \in \ZZ$, let $\bdr \otimes_F^{\sigma^n} D$ denote the tensor product of $\bdr$ and $D$ above $F$, where $F$ maps to $\bdr$ via $\sigma^n$. We then have $b \otimes a \cdot d = \sigma^n(a) \cdot b \otimes d$. Note that $\bdr \otimes_F^{\sigma^n} D$ only depends on $n \bmod{h}$. For each $0 \leq j \leq h-1$, let $\Fil_j^\bullet$ be a filtration on $D$ that is stable under $\Gamma_F$, and define $W_{\dR}^{+,j}(D) = \Fil_j^0(\bdr \otimes_F^{\sigma^j} D)$ so that $W_{\dR}^{+,j}$ is  a $G_F$-stable $\bdr^+$-lattice of $\bdr \otimes_F^{\sigma^j} D$.

\begin{exem}\label{exofd}
If $V$ is an $F$-linear crystalline representation of $G_F$ of dimension $d$, then $\dcris(V)$ is a free $F \otimes_{\Qp} F$-module of rank $d$ and we have 
\[ \dcris(V) = D \oplus \phi(D) \oplus \cdots \oplus \phi^{h-1}(D), \]
according to the decomposition of $F \otimes_{\Qp} F$ as $\prod_{\sigma^i : F \to F} F$. Each $\phi^j(D)$ comes with the filtration induced from $\dcris(V)$, and we set $\Fil_j^k D = \phi^{-j} (\Fil^k \dcris(V) \cap \phi^j(D))$. 
\end{exem}

We now briefly recall some definitions from \cite{ST03}. The ring $\calR^+(Y)$ is a Fr\'echet-Stein algebra; indeed, its topology is defined by the valuations $\{ W(\cdot,[0;s_n]) \}_{n \in S}$, where $S$ is any unbounded set of integers, and the ring $\calR^{[0;s_n]} (Y)$ is noetherian and flat over $\calR^{[0;s_m]} (Y)$ if $m \geq n \in S$. Recall that a coherent sheaf is then a family $\{ M^{[0;s_n]} \}_{n \in S}$ of finitely generated $\calR^{[0;s_n]} (Y)$-modules, such that $\calR^{[0;s_n]} (Y) \otimes_{\calR^{[0;s_m]} (Y)} M^{[0;s_m]} = M^{[0;s_n]}$ for all $m \geq n \in S$. A $\calR^+(Y)$-module $M$ is said to be coadmissible if $M$ is the set of global sections of a coherent sheaf $\{ M^{[0;s_n]} \}_{n \in S}$. We say that $M$ is a reflexive coadmissible $\calR^+(Y)$-module if each $M^{[0;s_n]}$ is a reflexive $\calR^{[0;s_n]} (Y)$-module. By lemma 8.4 of \cite{ST03}, this is the same as requiring that $M$ itself be a reflexive $\calR^+(Y)$-module.

Let $\lambda_j = \log_{\LT}(Y_j) / Y_j$ and $\lambda = \lambda_0 \cdots \lambda_{h-1}$, so that for any $n \in \ZZ$, $t$ is a $\Qp$-multiple of $\iota_n(\lambda \cdot Y_0 \cdots Y_{h-1})$. Let $f_j = \lambda / \lambda_j$, so that if $k \neq j \bmod{h}$, then $\iota_k(f_j)$ is a unit of $\bdr^+$.

If $y = \sum_i y_i \otimes d_i \in \calR^+(Y)[1/\lambda] \otimes_F D$, let $\iota_k(y) = \sum_i \iota_k(y_i) \otimes d_i \in \bdr \otimes_F^{\sigma^{-k}} D$. 

\begin{defi}\label{defmd}
Let $\mfont^+(D)$ be the set of $y \in \calR^+(Y)[1/\lambda] \otimes_F D$ that satisfy $\iota_k(y) \in W_{\dR}^{+,-k}(D)$ for all $k \geq h$.
\end{defi}

\begin{theo}\label{mprcs}
If $D$ is a $\phi_q$-module with an action of $\Gamma_F$ and $h$ filtrations, then
\begin{enumerate}
\item $\mfont^+(D)$ is a reflexive coadmissible $\calR^+(Y)$-module; 
\item the $\calR^+(Y)[1/f_j]$-module $\mfont^+(D)[1/f_j]$ is free of rank $d$ for $0 \leq j \leq h-1$;
\item $\mfont^+(D) = \cap_{j=0}^{h-1}  \mfont^+(D)[1/f_j]$.
\end{enumerate}
\end{theo}

In the remainder of this section, we prove theorem \ref{mprcs}. We now establish some preliminary results. Let $S = \{ hm+(h-1)$ where $m \geq 1 \}$, and take $n \in S$. Recall that on the ring $\calR^{[0;s_n]} (Y)$, the map $\iota_k$ is defined for $h \leq k \leq n$. Let 
\[ \mfont(D)^{[0;s_n]} = \{ y \in \calR^{[0;s_n]} (Y)[1/\lambda] \otimes_F D,\ \iota_k(y) \in W_{\dR}^{+,-k}(D) \text{ for all } h \leq k \leq n\}. \]
For $0 \leq j \leq h-1$, recall that $\calR^I(Y_j)$ is a ring of power series in one variable. Let
\begin{align*} 
N_j^{[0;s_n]} & = \{ y \in \calR^{[0;s_n]}(Y_j)[1/\lambda_j] \otimes_F D, \ \phi_q^{-k}\phi^{-j}(y) \in W_{\dR}^{+,-j}(D) \text{ for all } 1 \leq k \leq m\}, \\
N_j^+ & = \{ y \in \calR^+(Y_j)[1/\lambda_j] \otimes_F D, \ \phi_q^{-k}\phi^{-j}(y) \in W_{\dR}^{+,-j}(D) \text{ for all } k \geq 1 \}. 
\end{align*} 

\begin{prop}\label{mdonedim}
The $\calR^+ (Y_j)$-module $N_j^+$ is free of rank $d$, for all $n$ we have $N_j^{[0;s_n]} = \calR^{[0;s_n]}(Y_j) \otimes_{\calR^+(Y_j)} N_j^+$, and the map $\bdr^+ \otimes^{\phi_q^{-k}\phi^{-j}}_{\calR^+ (Y_j)} N_j^+ \to W_{\dR}^{+,-j}(D)$ is an isomorphism for all $k \geq 1$.
\end{prop}

\begin{proof}
Since there is only one variable, the proof is a standard argument, analogous to the one which one can find in \S II.1 of \cite{LB10} or \S 2.2 of \cite{KR09}.
\end{proof}

Let $M_j^{[0;s_n]} = \calR^{[0;s_n]}(Y)[1/f_j] \otimes_{\calR^{[0;s_n]}(Y_j)} N_j^{[0;s_n]}$, where $f_j = \lambda/\lambda_j$.

\begin{prop}\label{mintmj}
We have $\mfont(D)^{[0;s_n]} [1/f_j] = M_j^{[0;s_n]}$ and $\mfont(D)^{[0;s_n]} = \cap_j M_j^{[0;s_n]}$.
\end{prop}

\begin{proof}
In the sequel, we use the fact that $Q_1(Y_j) \cdots Q_m(Y_j)$ and $\lambda_j$ generate the same ideal of $\calR^{[0;s_n]}(Y_j)$ (recall that $n=hm+(h-1)$). Let $a$ and $b$ be two integers such that 
\[ t^a \cdot \bdr^+ \otimes_F^{\phi^j} D \subset W_{\dR}^{+,j}(D) \subset t^{-b} \cdot \bdr^+ \otimes_F^{\phi^j} D, \] for all $j$. We then have $\mfont(D)^{[0;s_n]} \subset \lambda^{-b} \cdot \calR^{[0;s_n]}(Y) \otimes_F D$ by theorem \ref{embdr}.

We have $\phi^{-(hk+j)}( \calR^{[0;s_n]} (Y)[1/f_j] ) \subset \bdr^+$ for all $1 \leq k \leq m$ so that if $y \in M_j^{[0;s_n]}$, then $\phi^{-(hk+j)}(y) \in W_{\dR}^{+,-j}(D)$ for all $1 \leq k \leq m$. On the other hand, if $y \in M_j^{[0;s_n]}$, then $y \in \lambda^{-c} \cdot \calR^{[0;s_n]}(Y) \otimes_F D$ for some $c \geq 0$, so that $f_j^{a+c} y \in \mfont(D)^{[0;s_n]}$. This implies that $M_j^{[0;s_n]} \subset \mfont(D)^{[0;s_n]}[1/f_j]$.

We now prove that $\mfont(D)^{[0;s_n]} \subset M_j^{[0;s_n]}$. Choose $y \in \mfont(D)^{[0;s_n]}$. Since 
\[ \mfont(D)^{[0;s_n]} \subset \lambda^{-b} \cdot \calR^{[0;s_n]}(Y) \otimes_F D, \]
we can write $y = \lambda^{-b} \sum_k z_k \otimes d_k$. By Weierstrass dividing (proposition \ref{weierdiv}) the $z_k$'s by the polynomial $(Q_1(Y_j) \cdots Q_m(Y_j))^{a+b}$, we can write $y=(Q_1(Y_j) \cdots Q_m(Y_j))^{a+b} z + y_0$ with $y_0 \in \calR^{[0;s_n]}(Y) [1/\lambda] \otimes_{\calR^{[0;s_n]}(Y_j)} N_j^{[0;s_n]}$. 

Note that $(Q_1(Y_j) \cdots Q_m(Y_j))^{a+b} z \in M_j^{[0;s_n]}$ because $t^a \bdr^+ \otimes_F^{\phi^j} D \subset W_{\dR}^{+,j}(D)$, so that $(Q_1(Y_j) \cdots Q_m(Y_j))^a  \cdot D \subset N_j^{[0;s_n]}$.

Write $y_0 = \sum_{k=1}^d a_k \otimes n_k $ where $a_k \in \calR^{[0;s_n]}(Y) [1/\lambda]$ and $n_1,\hdots,n_d$ is a basis of $N_j^{[0;s_n]}$. The element $y_0$ satisfies $\phi_q^{-\ell}\phi^{-j}(y_0) \in W_{\dR}^{+,-j}(D)$ for all $1 \leq \ell \leq m$. By proposition \ref{mdonedim}, the map 
\[ \bdr^+ \otimes^{\phi_q^{-\ell}\phi^{-j}}_{\calR^{[0;s_n]}(Y_j)} N_j^{[0;s_n]} \to W_{\dR}^{+,-j}(D) \] 
is an isomorphism; this implies that $\phi_q^{-\ell}\phi^{-j}(a_k) \in \bdr^+$ for all $1 \leq \ell \leq m$. Theorem \ref{embdr} now implies that $a_k$ has no pole at any of the roots of $Q_1(Y_j),\hdots,Q_m(Y_j)$, so that we have $a_k \in \calR^{[0;s_n]}(Y) [1/f_j]$. This implies that $y_0 \in M_j^{[0;s_n]}$, and therefore also $y$. This proves that $\mfont(D)^{[0;s_n]} \subset M_j^{[0;s_n]}$ and therefore $\mfont(D)^{[0;s_n]}[1/f_j] = M_j^{[0;s_n]}$. 

If $x \in \cap_j M_j^{[0;s_n]}$, and if $k = j \bmod{h}$ with $0 \leq j \leq h-1$, then the fact that $x \in \mfont(D)^{[0;s_n]}[1/f_j] = \calR^{[0;s_n]}(Y)[1/f_j] \otimes_{\calR^{[0;s_n]}(Y_j)} N_j^{[0;s_n]}$ implies that $\iota_k(x) \in W_{\dR}^{+,-k}(D)$. This is true for all $h \leq k \leq n$, so that $x \in \mfont(D)^{[0;s_n]}$ and this proves the second assertion.
\end{proof}

\begin{lemm}\label{mnten}
We have $\mfont^+(D)[1/f_j] = \calR^+(Y)[1/f_j] \otimes_{\calR^+(Y_j)} N_j^+$.
\end{lemm}

\begin{proof}
By combining propositions \ref{mdonedim} and \ref{mintmj}, we find that 
\[ \mfont(D)^{[0;s_n]} [1/f_j] = \calR^{[0;s_n]}(Y)[1/f_j] \otimes_{\calR^+(Y_j)} N_j^+. \]
Since $\mfont(D)^+ = \cap_j \mfont(D)^{[0;s_n]}$, we have $\mfont(D)^+ [1/f_j] \subset \cap_j \mfont(D)^{[0;s_n]} [1/f_j]$. We also have $\calR^+(Y)[1/f_j] \otimes_{\calR^+(Y_j)} N_j^+ \subset \mfont^+(D)[1/f_j]$, and those two inclusions imply that $\mfont^+(D)[1/f_j] = \calR^+(Y)[1/f_j] \otimes_{\calR^+(Y_j)} N_j^+$.
\end{proof}

\begin{proof}[Proof of theorem \ref{mprcs}]
We first prove that the family $\{ \mfont(D)^{[0;s_n]} \}_{n \in S}$ is a coherent sheaf. Take $n \geq m \in S$. We have 
\begin{multline*} 
\calR^{[0;s_m]}(Y) \otimes_{\calR^{[0;s_n]}(Y)} \mfont(D)^{[0;s_n]} \\ 
= \calR^{[0;s_m]}(Y)\otimes_{\calR^{[0;s_n]}(Y)} (\cap_j \calR^{[0;s_n]}(Y)[1/f_j] \otimes_{\calR^{[0;s_n]}(Y_j)} N_j^{[0;s_n]}) \\
= \cap_j \calR^{[0;s_m]}(Y)[1/f_j] \otimes_{\calR^{[0;s_n]}(Y_j)} N_j^{[0;s_n]}
= \mfont(D)^{[0;s_m]}.
\end{multline*}
This implies that the family $\{ \mfont(D)^{[0;s_n]} \}_{n \in S}$ is a coherent sheaf. It is clear that its global sections are precisely $\mfont^+(D)$. By proposition \ref{mintmj}, we have $\mfont(D)^{[0;s_n]} = \cap_j \mfont(D)^{[0;s_n]} [1/f_j]$ where each $\mfont(D)^{[0;s_n]} [1/f_j]$ is free of rank $d$ over $\calR(Y)^{[0;s_n]} [1/f_j]$. The fact that $\mfont(D)^{[0;s_n]}$ is reflexive now follows from proposition 6 of VII.4.2 of \cite{AC61}, and this proves (1).

By combining proposition \ref{mdonedim} and lemma \ref{mnten}, we get item (2) of the theorem. Suppose now that $x \in \cap_j \mfont^+(D)[1/f_j]$. If $k = j \bmod{h}$ with $0 \leq j \leq h-1$, then the fact that $x \in \mfont^+(D)[1/f_j] = \calR^+(Y)[1/f_j] \otimes_{\calR^+(Y_j)} N_j^+$ implies that $\iota_k(x) \in W_{\dR}^{+,-k}(D)$. This being true for all $k \geq h$, we have $x \in \mfont^+(D)$ and this proves item (3) of the theorem.
\end{proof}

\begin{rema}
\label{h2free}
If $h \leq 2$, then the ring $\calR^{[0;s_n]}(Y)$ is of dimension $\leq 2$, and reflexive $\calR^{[0;s_n]}(Y)$-modules are therefore projective. By L\"utkebohmert's theorem (see \cite{L77}, corollary on page 128), the $\calR^{[0;s_n]}(Y)$-module $\mfont(D)^{[0;s_n]}$ is then free of rank $d$. The system $\{ \mfont(D)^{[0;s_n]}\}_{n \in S}$ then forms a vector bundle over the open unit polydisk. By combining proposition 2 on page 87 of \cite{LG68} (note that ``$A_m$'' is defined at the bottom of page 82 of loc.\ cit.), and the main theorem of \cite{WB81}, we get that $\mfont^+(D)$ is free of rank $d$ over $\calR^+(Y)$. If $h \geq 3$, I do not know whether this still holds.
\end{rema}

\section{Properties of $\mfont^+(D)$}\label{propsmd}

We now prove that $\mfont(D) = \calR(Y) \otimes_{\calR^+(Y)} \mfont^+(D)$ is a $(\phi_q,\Gamma_F)$-module over $\calR(Y)$, and that if $D$ arises from a crystalline representation $V$, then $\mfont^+(D)$ and $V$ are naturally related. It is clear from the definition that $\mfont^+(D)$ is stable under the action of $\Gamma_F$. We also have $\lambda^a \cdot \calR^+(Y) \otimes_F D \subset \mfont^+(D)$ for some $a \geq 0$, so that
\[ \calR^+(Y)[1/\lambda] \otimes_{\calR^+(Y)} \mfont^+(D)  = \calR^+(Y)[1/\lambda] \otimes_F D. \]

Say that the module $D$ with $h$ filtrations is effective if $\Fil^0_j(D)=D$ for $0 \leq j \leq h-1$. Recall that $n = hm+(h-1)$ with $m \geq 1$. 

\begin{lemm}\label{phinj}
If $D$ is effective, then the $\calR^+(Y_j)$-module $N_j^+$ is stable under $\phi_q$, and $N_j^+ / \phi_q^*(N_j^+)$ is killed by $Q_1(Y_j)^{a_j}$ if $a_j \geq 0$ is such that $\Fil^{a_j+1} D = \{ 0 \}$.
\end{lemm}

\begin{proof}
This concerns the construction in one variable, so the proof is standard. See for example \S 2.2 of \cite{KR09}.
\end{proof}

\begin{prop}\label{phionmd}
If $D$ is effective, then the $\calR^+(Y)$-module $\mfont^+(D)$ is stable under the Frobenius map $\phi_q$, and $\mfont^+(D) / \phi_q^* (\mfont^+(D))$ is killed by $Q_1(Y_0)^{a_0} \cdots Q_1(Y_{h-1})^{a_{h-1}}$.
\end{prop}

\begin{proof}
By (2) of theorem \ref{mprcs}, we have $\mfont^+(D) = \cap_j \mfont^+(D)[1/f_j]$ and by lemma \ref{mnten}, $\mfont^+(D)[1/f_j] = \calR^+(Y)[1/f_j] \otimes_{\calR^+(Y_j)} N_j^+$. Lemma \ref{phinj} implies that $N_j^+$ is stable under $\phi_q$, and so the same is true of $\mfont^+(D)[1/f_j]$ and hence $\mfont^+(D)$.

If $x \in \mfont^+(D)$, then $x \in \mfont^+(D)[1/f_j] = \calR^+(Y)[1/f_j] \otimes_{\calR^+(Y_j)} N_j^+$. Note however that at each $k = i \neq j \mod{h}$, the coefficients of $x$ can have a pole of order at most $a_i$ since $\Fil^{a_i+1} D = \{ 0 \}$. This implies the more precise estimate 
\[ \mfont^+(D) \subset \prod_{i \neq j} \lambda_i^{-a_i} \cdot \calR^+(Y) \otimes_{\calR^+(Y_j)} N_j^+. \]
The $\phi_q(\calR^+(Y))$-module $\calR^+(Y)$ is free of rank $q^h$, with basis $\{ Y^\ell$, $\ell \in \{ 0, \hdots, q-1 \}^h\}$. We therefore have
\begin{align*} 
Q_1(Y_0)^{a_0} \cdots Q_1(Y_{h-1})^{a_{h-1}} \cdot x  & \in \prod_{i \neq j}  (\lambda_i/Q_1(Y_i)) ^{-a_i} \cdot \calR^+(Y) \otimes_{\calR^+(Y_j)} Q_1(Y_j)^{a_j} \cdot N_j^+ \\ 
& \subset \oplus_\ell Y^\ell \cdot \phi_q( \calR^+(Y)[1/f_j] \otimes_{\calR^+(Y_j)} N_j^+ ).
\end{align*}
This implies that 
\begin{align*} 
Q_1(Y_0)^{a_0} \cdots Q_1(Y_{h-1})^{a_{h-1}} \cdot x \in \cap_j \oplus_\ell Y^\ell \cdot \phi_q(\mfont^+(D)[1/f_j]) = \phi_q^*( \mfont^+(D)),
\end{align*}
which proves the second claim.
\end{proof}

\begin{rema}\label{moregend}
Instead of working with a $D$ where the filtrations are defined on $D$, we could have asked for the filtrations to be defined on $F_n \otimes_F D$ for some $n \geq 1$. The construction and properties of $\mfont^+(D)$ are then basically unchanged, but the annihilator of $\mfont^+(D) / \phi_q^* (\mfont^+(D))$ is possibly more complicated than in proposition \ref{phionmd}. This applies in particular to representations of $G_F$ that become crystalline when restricted to $G_{F_n}$ for some $n \geq 1$.
\end{rema}

\begin{defi}\label{defpgm}
A $(\phi_q,\Gamma_F)$-module over $\calR(Y)$ is a $\calR(Y)$-module $\mfont$ that is of the form $\mfont = \calR(Y) \otimes_{\calR^{[s;+\infty[}(Y)} \mfont^{[s;+\infty[}$ where $\mfont^{[s;+\infty[}$ is a coadmissible $\calR^{[s;+\infty[}(Y)$-module, endowed with a semilinear Frobenius map $\phi_q : \mfont^{[s;+\infty[} \to \mfont^{[qs;+\infty[}$, such that $\phi_q^*(\mfont^{[s;+\infty[}) = \mfont^{[qs;+\infty[}$, and a continuous and compatible action of $\Gamma_F$.
\end{defi}

\begin{rema}\label{defpgmrem}
In the definition above, it would seem natural to impose some additional condition on $\mfont$, such as ``torsion-free''. All the $(\phi_q,\Gamma_F)$-modules over $\calR(Y)$ that are constructed in this article are actually reflexive. The definition above should be considered provisional, until we have a better idea of which objects we want to exclude. Note that in the absence of flatness, tensor products may behave badly.
\end{rema}

If $D$ is a $\phi_q$-module with an action of $\Gamma_F$ and $h$ filtrations and if $\ell \in \ZZ$, let $D(\ell)$ denote the same $\phi_q$-module with an action of $\Gamma_F$, but with $\Fil^k_j(D(\ell)) = (\Fil^{k-\ell} D )(\ell)$. Note that $D(\ell)$ is effective if $\ell \gg 0$.

\begin{lemm}\label{twmd}
We have $\mfont(D(\ell)) = \lambda^{-\ell} \cdot \mfont(D)$.
\end{lemm}

\begin{proof}
The fact that $\mfont^+(D(\ell)) = \lambda^{-\ell} \cdot \mfont^+(D)$ follows from the definition.
\end{proof}

\begin{theo}\label{thbone}
If $D$ is a $\phi_q$-module with an action of $\Gamma_F$ and $h$ filtrations as above, then $\calR(Y) \otimes_{\calR^+(Y)} \mfont^+(D)$ is a $(\phi_q,\Gamma_F)$-module over $\calR(Y)$.
\end{theo}

\begin{proof}
If $D$ is effective, then this follows from theorem \ref{mprcs} and proposition \ref{phionmd}. If $D$ is not effective, then $D(\ell)$ is effective if $\ell \gg 0$, and the theorem follows from the effective case and lemma \ref{twmd}.
\end{proof}

\begin{rema}\label{linkisin}
In \cite{KR09}, Kisin and Ren construct some $(\phi_q,\Gamma_F)$-modules $\mfont^+_{\mathrm{KR}}(D)$ in one variable, over the ring $\calR^+(Y_0)$, from the data of a $D$ like ours for which the filtration $\Fil_j^\bullet$ is trivial for $j\neq 0$. For those $D$, we have $\mfont^+(D) = \calR^+(Y) \otimes_{\calR^+(Y_0)} \mfont^+_{\mathrm{KR}}(D)$. More generally, our construction shows that $\mfont^+(D)$ comes by extension of scalars from a $(\phi_q,\Gamma_F)$-module in as many variables as there are nontrivial filtrations among the $\Fil_j^\bullet$.
\end{rema}

\begin{prop}\label{surjlatt}
If $n=hk+j \geq h$, then the map 
\[ \bdr^+ \otimes_{\iota_n(\calR^+(Y))} \iota_n(\mfont^+(D)) \to \Fil_{-j}^0(\bdr \otimes_F^{\sigma^{-j}} D) \] 
is an isomorphism.
\end{prop}

\begin{proof}
We have $\bdr^+ \otimes_{\iota_n(\calR^+(Y_j))} \iota_n(N_j^+) = \Fil_{-j}^0(\bdr \otimes_F^{\sigma^{-j}} D)$, for example by lemma II.1.5 of \cite{LB10}. If $x \in \calR^+(Y) \otimes_{\calR^+(Y_j)} N_j^+$, then there exists $c \gg 0$ such that $f_j^c x \in \mfont^+(D)$, and the lemma results from the fact that $\iota_n(f_j)$ is a unit of $\bdr^+$.
\end{proof}

Suppose now that $D$ comes from an $F$-linear crystalline representation $V$ of $G_F$ as in example \ref{exofd}. In this case, $\Fil_j^0(\bdr \otimes_F^{\sigma^j} D) = \bdr^+ \otimes_{F}^{\phi^j} V$. Moreover, one recovers $V$ from $D$ by the formula: 
\[ V = \{ y \in (\btrigplus[1/t] \otimes_F D)^{\phi_q=1},\ y \in \Fil_j^0(\bdr \otimes_F^{\sigma^j} D) \text{ for all } 0 \leq j \leq h-1\}. \]

Recall that we have constructed in \S\ref{iotanlt} an injective map $\calR^+(Y) \to \btrigplus$. This way we get a map 
\[ \btrigplus \otimes_{\calR^+(Y)} \mfont^+(D) \to \btrigplus[1/t] \otimes_F D \to \btrigplus[1/t] \otimes_F V. \]

Let $\btrig{,r}{}$ be the rings defined in \S 2.3 \cite{LB2}. Recall that $n(r)$ is the smallest $n$ such that $r \leq p^{n-1}(p-1)$. We have the following lemma.

\begin{lemm}\label{kot}
If $y \in \btrig{,r}{}[1/t]$ satisfies $\phi^{-n}(y) \in \bdr^+$ for all $n \geq n(r)$, then $y \in \btrig{,r}{}$.
\end{lemm}

\begin{proof}
See lemma 1.1 of \cite{LB12} and the proof of proposition 3.2 in ibid.
\end{proof}

\begin{theo}\label{btrisom}
If $D$ comes from a crystalline representation $V$, and if $r \geq p^{h-1}(p-1)$, then the map above gives rise to an isomorphism 
\[ \btrig{,r}{} \otimes_{\calR^+(Y)} \mfont^+(D) \to \btrig{,r}{} \otimes_F V. \]
\end{theo}

\begin{proof}
We first check that the image of the map above belongs to $\btrig{,r}{} \otimes_F V$. If $y \in \btrig{,r}{} \otimes_{\calR^+(Y)} \mfont^+(D)$, then its image is in $\btrig{,r}{}[1/t] \otimes_F V$ and satisfies $\phi^{-n}(y) \in \bdr^+ \otimes_F^{\sigma^{-n}} V$ for all $n \geq n(r)$, so the assertion follows from lemma \ref{kot}.

We now prove that $\btrig{,r}{} \otimes_{\calR^+(Y)} \mfont^+(D)$ is a free $\btrig{,r}{}$-module of rank $d$. By (2) of theorem \ref{mprcs}, $\mfont^+(D)[1/f_j]$ is a free $\calR^+(Y)[1/f_j]$-module of rank $d$, and therefore $\btrig{,r}{}[1/f_j] \otimes_{\calR^+(Y)} \mfont^+(D)$ is a free $\btrig{,r}{}[1/f_j]$-module of rank $d$ for all $j$. The ring $\btrig{,r}{}$ is a B\'ezout ring by theorem 2.9.6 of \cite{KSF}, and the elements $f_0,\hdots,f_{h-1}$ have no common factor. They therefore generate the unit ideal of $\btrig{,r}{}$, and  $\btrig{,r}{} \otimes_{\calR^+(Y)} \mfont^+(D)$ is projective of rank $d$ by theorem 1 of II.5.2 of \cite{AC61}. Since $\btrig{,r}{}$ is a B\'ezout ring, $\btrig{,r}{} \otimes_{\calR^+(Y)} \mfont^+(D)$ is free of rank $d$. By proposition \ref{surjlatt}, the map 
\[ \bdr^+ \otimes_{\iota_n(\btrig{,r}{})} \iota_n(\btrig{,r}{} \otimes_{\calR^+(Y)} \mfont^+(D)) \to \bdr^+ \otimes_{F}^{\phi^{-n}} V \] 
is an isomorphism if $n \geq n(r)$. The two $\btrig{,r}{}$-modules $\btrig{,r}{} \otimes_{\calR^+(Y)} \mfont^+(D)$ and $\btrig{,r}{} \otimes_F V$ therefore have the same localizations at all $n \geq n(r)$, and are both stable under $G_F$, so that they are equal by the same argument as in the proof of lemma 2.2.2 of \cite{LB8} (the idea is to take determinants, so that one is reduced to showing that if $x \in \btrig{,r}{}$ generates an ideal stable under $G_F$, and has the property that $\iota_n(x)$ is a unit of $\bdr^+$ for all $n \geq n(r)$, then $x$ is a unit of $\btrig{,r}{}$).
\end{proof}

\begin{rema}\label{btrisomphin}
If $D$ comes from a crystalline representation $V$, and if $n \geq 0$, then there is likewise an isomorphism $\btrig{,r}{} \otimes^{\phi^{-n}}_{\calR^+(Y)} \mfont^+(D) \to \btrig{,r}{} \otimes^{\sigma^{-n}}_F V$ for $r \gg 0$.
\end{rema}

\section{Crystalline $(\phi_q,\Gamma_F)$-modules}\label{pmfedp}

Let $\mfont$ be a $(\phi_q,\Gamma_F)$-module over $\calR(Y)$. In this section, we define what it means for $\mfont$ to be crystalline, and we prove that every crystalline $(\phi_q,\Gamma_F)$-module $\mfont$ is of the form $\mfont=\mfont(D)$, where $D$ is a filtered $\phi_q$-module on which the action of $G_F$ is trivial. The results are similar to those of \cite{LB10}, which deals with the cyclotomic case.

\begin{lemm}\label{invfro}
We have $\Frac(\calR(Y))^{\Gamma_F}=F$.
\end{lemm}

\begin{proof}
If $x  \in \Frac(\calR(Y))^{\Gamma_F}$, then we can write $x=a/b$ with $a, b \in \calR^{[s_n;s_n]}(Y)$ for some $n \gg 0$. By proposition \ref{embti}, the series $a(u,\hdots,\phi^{h-1}(u))$ and $b(u,\hdots,\phi^{h-1}(u))$ converge in $\bt^{[r_n;r_n]}$. We can therefore see $\phi^{-n}(a)$ and $\phi^{-n}(b)$ as elements of $\bdr^+$, which satisfy $\phi^{-n}(a) / \phi^{-n}(b) \in \bdr^{G_F}$. The lemma now follows from the fact that $\bdr^{G_F} = F$.
\end{proof}

If $\mfont$ is a $(\phi_q,\Gamma_F)$-module over $\calR(Y)$, then let $\dcris(\mfont) =  (\calR(Y)[1/t] \otimes_{\calR(Y)} \mfont)^{\Gamma_F}$.

\begin{coro}\label{maxdiminv}
If $\mfont$ is a $(\phi_q,\Gamma_F)$-module over $\calR(Y)$, then $\dim \dcris(\mfont) \leq \dim \Frac(\calR(Y)) \otimes_{\calR(Y)} \mfont$.
\end{coro}

\begin{proof}
By a standard argument, lemma \ref{invfro} implies that the map 
\[ \Frac(\calR(Y)) \otimes_F \dcris(V) \to \Frac(\calR(Y)) \otimes_{\calR(Y)} \mfont \] is injective.
\end{proof}

\begin{defi}\label{defpgcr}
We say that a $(\phi_q,\Gamma_F)$-module $\mfont$ over $\calR(Y)$ is crystalline if 
\begin{enumerate}
\item for some $s$, $\mfont^{[s;+\infty[} [1/f_j]$ is a free $\calR(Y)^{[s;+\infty[} [1/f_j]$-module of finite rank $d$;
\item $\mfont^{[s;+\infty[} = \cap_{j=0}^{h-1} \mfont^{[s;+\infty[} [1/f_j]$;
\item we have $\dim \dcris(\mfont) = d$.
\end{enumerate}
\end{defi}

For example, if $D$ is a filtered $\phi_q$-module on which the action of $G_F$ is trivial, then $\mfont(D)$ is a crystalline $(\phi_q,\Gamma_F)$-module.

\begin{prop}
\label{idealgam}
If $f \in \calR^{[s;+\infty[}(Y)$ generates an ideal of $\calR^{[s;+\infty[}(Y)$ that is stable under $\Gamma_F$, then $f = u \cdot \prod_{j=0}^{h-1} \prod_{n \geq n(s)} (Q_n(Y_j)/p)^{a_{n,j}}$ where $u$ is a unit and $a_{n,j} \in \ZZ_{\geq 0}$.
\end{prop}

\begin{proof}
Recall that a power series $f \in \calR^I(Y)$ is a unit if and only if it has no zero in the corresponding domain of convergence (by the nullstellensatz, see \S 7.1.2 of \cite{BGR}).

Let $I=[s;u]$ be a closed subinterval of $[s;+\infty[$, so that $f \in \calR^I(Y)$, and let $z=(z_0,z_1,\hdots,z_{h-1})$ be a point such that $f(z)=0$. Let $J$ be the set of indices $j$ such that $z_j$ is not a torsion point of $\LT_h$ and let $f_J \in \calR^I_{F_k}(\{Y_j\}_{j \in J})$ be the power series that results from evaluation of the $Y_m$ at $z_m$ for all the $z_m$ that are torsion points of $\LT_h$ (here $k$ is large enough so that all those $z_m$ belong to $F_k$). The ideal of $\calR^I_{F_k}(\{Y_j\}_{j \in J})$ generated by the power series $f_J$ is stable under $1+p^k \OO_F$, so that the set of its zeroes is stable under the action of $1+p^k \OO_F$. Furthermore, $f_J$ has a zero none of whose coordinates are torsion points of $\LT_h$. The same argument as in the proof of proposition \ref{injeval} shows that $f_J=0$.

If we denote by $Z_I(f)$ the zero set of $f \in \calR^I(Y)$, then the preceding argument shows that $Z_I(f)$ is the union of finitely many components of the form $Z_0 \times \cdots Z_{h-1}$ where for each $j$, either $Z_j$ is a torsion point of $\LT_h$ or $Z_j=Z_I(\{0\})$. For reasons of dimension, each of these components has precisely one $Z_j$ which is a torsion point, the remaining $h-1$ being $Z_I(\{0\})$. This implies that in $\calR^I(Y)$, $f$ is the product of finitely many $Q_n(Y_j)$ by a unit. 

The proposition now follows by a standard infinite factorisation argument, by writing $[s;+\infty[ = \cup_{u\geq s} [s;u]$.
\end{proof}

\begin{coro}\label{rstdm}
If $\mfont$ is a crystalline $(\phi_q,\Gamma_F)$-module over $\calR(Y)$, then the map
\[ \calR(Y)[1/t] \otimes_F \dcris(\mfont) \to \calR(Y)[1/t] \otimes_{\calR(Y)} \mfont \]
is an isomorphism.
\end{coro}

\begin{proof}
The map is injective by lemma \ref{invfro}, and its determinant generates an ideal of $\calR(Y)[1/t]$ that is stable under $\Gamma_F$. Proposition \ref{idealgam} implies that this ideal is the unit ideal of $\calR(Y)[1/t]$, and therefore that the map is an isomorphism.
\end{proof}

We now consider filtrations on $\dcris(\mfont)$.

\begin{lemm}\label{misgam}
Let $D$ be an $F$-vector space with a trivial action of $G_F$, and let $W$ be a $\bdr^+$-lattice of $\bdr \otimes_F D$ that is stable under $G_F$. If we set $\Fil^i D =  D \cap t^i \cdot W$, then $W = \Fil^0 (\bdr \otimes_F D)$.
\end{lemm}

\begin{proof}
Let $e_1,\hdots,e_d$ be a basis of $D$ adapted to its filtration, with $e_i \in \Fil^{h_i} \setminus \Fil^{h_i+1} D$. We then have $\Fil^0 (\bdr \otimes_F D) = \oplus_{i=1}^d \bdr^+ \cdot t^{-h_i} e_i$. By definition, we have $t^{-h_i} e_i \in W$, so that $\Fil^0 (\bdr \otimes_F D) \subset W$. We now prove the reverse inclusion. 

If $w \in W$, then we can write $w=a_1 t^{-h_1} e_1 + \cdots + a_d t^{-h_d} e_d$ with $a_i \in \bdr$ and we need to prove that $a_i \in \bdr^+$ for all $i$. If this is not the case, then there exists $n \geq 1$ such that if we set $b_i = t^n a_i$, then we have $b_1 t^{-h_1} e_1 + \cdots + b_d t^{-h_d} e_d \in t \cdot W$, with $b_i \in (\bdr^+)^\times$ for at least one $i$. Consider the shortest such relation; in particular, $b_i \in (\bdr^+)^\times$ for all $i$ for which $b_i \neq 0$, and we can assume that $b_i=1$ for at least one $i$. If $g \in G_F$, then applying $1-\chi_{\cycl}(g)^{h_i} g$ to the relation yields a shorter relation. This implies that $(1-\chi_{\cycl}(g)^{h_i-h_j} g)(b_j) \in t \bdr^+$ for all $g \in G_F$ and all $1 \leq j \leq d$. Since $H^0(G_F,\Cp)=F$ and $H^0(G_F,\Cp(h))=\{0\}$ if $h \neq 0$, we have $b_j \in F + t\bdr^+$ if $h_i=h_j$ and $b_j \in t\bdr^+$ otherwise. The relation above therefore reduces to an $F$-linear combination of those $e_j$ for which $h_j=h_i$, belonging to $D \cap t^{h_i+1} W = \Fil^{h_i+1} D$, and is hence trivial. This proves that $W \subset \Fil^0 (\bdr \otimes_F D)$.
\end{proof}

\begin{defi}
\label{filmf}
Let $\mfont$ be a crystalline $(\phi_q,\Gamma_F)$-module over $\calR(Y)$, which we can write as $\mfont = \calR(Y) \otimes_{\calR^{[s;+\infty[}(Y)} \mfont^{[s;+\infty[}$ for some $s$ large enough. For $m \gg 0$ and $j=0,\hdots,h-1$ and $n=hm-j$, define 
\[ \Fil^i_j(\dcris(\mfont)) = \dcris(\mfont) \cap t^i \cdot (\bdr^+ \otimes_{\calR^{[s;+\infty[}(Y)}^{\phi^{-n}} \mfont^{[s;+\infty[}). \]
\end{defi}

\begin{prop}
\label{filok}
The definition of $\Fil^i_j(\dcris(\mfont))$ does not depend on $m \gg 0$, and we have $\Fil^0(\bdr \otimes^{\sigma^{-n}}_F \dcris(\mfont)) = \bdr^+ \otimes^{\phi^{-n}}_{\calR^{[s;+\infty[}(Y)} \mfont^{[s;+\infty[}$.
\end{prop}

\begin{proof}
If $s$ is large enough, then $\mfont^{[qs;+\infty[} = \phi_q^*(\mfont^{[s;+\infty[})$ so that 
\[Ê\bdr^+ \otimes^{\phi^{-n-h}}_{\calR^{[qs;+\infty[}(Y)} \mfont^{[qs;+\infty[} 
= \bdr^+ \otimes^{\phi^{-n}\phi_q^{-1}}_{\calR^{[qs;+\infty[}(Y)} \phi_q(\mfont^{[s;+\infty[})
= \bdr^+ \otimes^{\phi^{-n}}_{\calR^{[s;+\infty[}(Y)} \mfont^{[s;+\infty[}, \]
which implies the first statement. The second statement follows from lemma \ref{misgam}, applied to $W=\bdr^+ \otimes^{\phi^{-n}}_{\calR^{[s;+\infty[}(Y)} \mfont^{[s;+\infty[}$. 
\end{proof}

\begin{theo}
\label{crysequiv}
The functors $\mfont \mapsto \dcris(\mfont)$ and $D \mapsto \mfont(D)$, between the category of crystalline $(\phi_q,\Gamma_F)$-modules over $\calR(Y)$ and the category of $\phi_q$-modules with $h$ filtrations, are mutually inverse.
\end{theo}

\begin{proof}
If $D$ is a $\phi_q$-module with $h$ filtrations, then it is clear that $\dcris(\mfont(D))=D$ as $\phi_q$-modules. The fact that $\Fil^i_j(D) = D \cap t^i \cdot \Fil_j^0 (\bdr \otimes_F^{\sigma^{-n}} D)$ follows from taking a basis of $D$ adapted to $\Fil^\bullet_j$ and 
\[ \Fil_j^0 (\bdr \otimes_F^{\sigma^{-n}} D) = \bdr^+ \otimes^{\phi^{-n}}_{\calR^{[s;+\infty[}(Y)} \mfont^{[s;+\infty[}(D) = \Fil_j^0 (\bdr \otimes_F^{\sigma^{-n}} \dcris(\mfont(D))) \]
by propositions \ref{surjlatt} and \ref{filok}, so that the filtrations on $D$ and $\dcris(\mfont)$ are the same.

We now check that if $\mfont$ is a crystalline $(\phi_q,\Gamma_F)$-module over $\calR(Y)$ and $D=\dcris(\mfont)$ with the filtration given in definition \ref{filmf}, then $\mfont=\mfont(D)$. Corollary \ref{rstdm} says that we have $\calR(Y)[1/t] \otimes_F D =  \calR(Y)[1/t] \otimes_{\calR(Y)} \mfont$. The theorem now follows from proposition \ref{filok} and the fact that if $y \in \calR(Y)[1/t] \otimes_{\calR(Y)} \mfont$, then $y \in \mfont$ if and only if $y \in \bdr^+ \otimes^{\phi^{-n}}_{\calR^{[s;+\infty[}(Y)} \mfont$ for all $n \gg 0$ by theorem \ref{embdr} and items (1) and (2) of definition \ref{defpgcr}.
\end{proof}

\section*{Acknowledgements} 
I am grateful to P.\ Colmez, J.-M.\ Fontaine, L.\  Fourquaux, M.\ Gros, K.\  Kedlaya, R.\  Liu, V.\ Pilloni, S.\ Rozensztajn, P.\ Schneider, B.\ Stroh, B.\  Xie, S.\ Zerbes and especially C.\ Breuil for helpful conversations and remarks. Special thanks to P.\ Schneider and the referee for pointing out several embarassing mistakes in a previous version of this paper and helping me to correct them.

\bibliographystyle{smfalpha}
\bibliography{PGMLT_MRLv2}

\providecommand{\bysame}{\leavevmode ---\ }
\providecommand{\og}{``}
\providecommand{\fg}{''}
\providecommand{\smfandname}{\&}
\providecommand{\smfedsname}{\'eds.}
\providecommand{\smfedname}{\'ed.}
\providecommand{\smfmastersthesisname}{M\'emoire}
\providecommand{\smfphdthesisname}{Th\`ese}
\begin{thebibliography}{BGR84}

\bibitem[Bar81]{WB81}
{\scshape W.~Bartenwerfer} -- {\og {$k$}-holomorphe {V}ektorraumb\"undel auf
  offenen {P}olyzylindern\fg}, \emph{J. Reine Angew. Math.} \textbf{326}
  (1981), p.~214--220.

\bibitem[BC]{STLAV}
{\scshape L.~Berger {\normalfont \smfandname} P.~Colmez} -- {\og Sen theory and
  locally analytic vectors\fg}, In preparation.

\bibitem[Ber02]{LB2}
{\scshape L.~Berger} -- {\og Repr{\'e}sentations {$p$}-adiques et \'equations
  diff\'erentielles\fg}, \emph{Invent. Math.} \textbf{148} (2002), no.~2,
  p.~219--284.

\bibitem[Ber08a]{LB8}
\bysame , {\og Construction de {$(\varphi,\Gamma)$}-modules:
  repr{\'e}sentations {$p$}-adiques et {$B$}-paires\fg}, \emph{Algebra \&
  Number Theory} \textbf{2} (2008), no.~1, p.~91--120.

\bibitem[Ber08b]{LB10}
\bysame , {\og Equations diff{\'e}rentielles {$p$}-adiques et
  {$(\varphi,{N})$}-modules filtr{\'e}s\fg}, \emph{Ast\'erisque} (2008),
  no.~319, p.~13--38.

\bibitem[Ber09]{LB12}
\bysame , {\og Presque {$\mathbf{C}_p$}-repr{\'e}sentations et
  {$(\varphi,\Gamma)$}-modules\fg}, \emph{J. Inst. Math. Jussieu} \textbf{8}
  (2009), no.~4, p.~653--668.

\bibitem[BGR84]{BGR}
{\scshape S.~Bosch, U.~G{\"u}ntzer {\normalfont \smfandname} R.~Remmert} --
  \emph{Non-{A}rchimedean analysis}, Grundlehren der Mathematischen
  Wissenschaften, vol. 261, Springer-Verlag, Berlin, 1984.

\bibitem[Bou61]{AC61}
{\scshape N.~Bourbaki} -- \emph{\'{E}l\'ements de math\'ematique. {A}lg\`ebre
  commutative.}, Actualit\'es Scientifiques et Industrielles, Hermann, Paris,
  1961.

\bibitem[Bre10]{BR10}
{\scshape C.~Breuil} -- {\og The emerging {$p$}-adic {L}anglands programme\fg},
  in \emph{Proceedings of the {I}nternational {C}ongress of {M}athematicians.
  {V}olume {II}} (New Delhi), Hindustan Book Agency, 2010, p.~203--230.

\bibitem[Col02]{C02}
{\scshape P.~Colmez} -- {\og Espaces de {B}anach de dimension finie\fg},
  \emph{J. Inst. Math. Jussieu} \textbf{1} (2002), no.~3, p.~331--439.

\bibitem[Col10]{PC10}
\bysame , {\og {$(\phi,\Gamma)$}-modules et repr\'esentations du mirabolique de
  {${\mathrm{GL}}_2(\mathbf{Q}_p)$}\fg}, \emph{Ast\'erisque} (2010), no.~330,
  p.~61--153.

\bibitem[Fon90]{F90}
{\scshape J.-M. Fontaine} -- {\og Repr\'esentations {$p$}-adiques des corps
  locaux. {I}\fg}, in \emph{The {G}rothendieck {F}estschrift, {V}ol.\ {II}},
  Progr. Math., vol.~87, Birkh\"auser Boston, Boston, MA, 1990, p.~249--309.

\bibitem[Fon94]{FPP}
\bysame , {\og Le corps des p\'eriodes {$p$}-adiques\fg}, \emph{Ast\'erisque}
  (1994), no.~223, p.~59--111, With an appendix by Pierre Colmez, P{\'e}riodes
  $p$-adiques (Bures-sur-Yvette, 1988).

\bibitem[FW79]{FW}
{\scshape J.-M. Fontaine {\normalfont \smfandname} J.-P. Wintenberger} -- {\og
  Le ``corps des normes'' de certaines extensions alg\'ebriques de corps
  locaux\fg}, \emph{C. R. Acad. Sci. Paris S\'er. A-B} \textbf{288} (1979),
  no.~6, p.~A367--A370.

\bibitem[FX12]{FX12}
{\scshape L.~Fourquaux {\normalfont \smfandname} B.~Xie} -- {\og Triangulable
  {$\mathcal{O}_F$}-analytic {$(\varphi_q,\Gamma)$}-modules of rank $2$\fg},
  preprint, 2012.

\bibitem[Gru68]{LG68}
{\scshape L.~Gruson} -- {\og Fibr\'es vectoriels sur un polydisque
  ultram\'etrique\fg}, \emph{Ann. Sci. \'Ecole Norm. Sup. (4)} \textbf{1}
  (1968), p.~45--89.

\bibitem[Ked05]{KSF}
{\scshape K.~S. Kedlaya} -- {\og Slope filtrations revisited\fg}, \emph{Doc.
  Math.} \textbf{10} (2005), p.~447--525 (electronic).

\bibitem[KR09]{KR09}
{\scshape M.~Kisin {\normalfont \smfandname} W.~Ren} -- {\og Galois
  representations and {L}ubin-{T}ate groups\fg}, \emph{Doc. Math.} \textbf{14}
  (2009), p.~441--461.

\bibitem[L{\"u}t77]{L77}
{\scshape W.~L{\"u}tkebohmert} -- {\og Vektorraumb\"undel \"uber
  nichtarchimedischen holomorphen {R}\"aumen\fg}, \emph{Math. Z.} \textbf{152}
  (1977), no.~2, p.~127--143.

\bibitem[ST01]{ST01}
{\scshape P.~Schneider {\normalfont \smfandname} J.~Teitelbaum} -- {\og
  {$p$}-adic {F}ourier theory\fg}, \emph{Doc. Math.} \textbf{6} (2001),
  p.~447--481 (electronic).

\bibitem[ST03]{ST03}
\bysame , {\og Algebras of {$p$}-adic distributions and admissible
  representations\fg}, \emph{Invent. Math.} \textbf{153} (2003), no.~1,
  p.~145--196.

\bibitem[Win83]{WCN}
{\scshape J.-P. Wintenberger} -- {\og Le corps des normes de certaines
  extensions infinies de corps locaux; applications\fg}, \emph{Ann. Sci.
  \'Ecole Norm. Sup. (4)} \textbf{16} (1983), no.~1, p.~59--89.

\bibitem[Z{\'a}b12]{SZ}
{\scshape G.~Z{\'a}br{\'a}di} -- {\og Generalized {R}obba rings\fg},
  \emph{Israel J. Math.} \textbf{191} (2012), p.~817--887, With an appendix by
  Peter Schneider.

\end{thebibliography}
\end{document}